\documentclass[reqno,12pt,letterpaper]{amsart}
\usepackage{amsmath,amssymb,amsthm,graphicx,mathrsfs,url}
\usepackage[usenames,dvipsnames]{color}
\usepackage[colorlinks=true,linkcolor=Red,citecolor=Green]{hyperref}
\usepackage[super]{nth}
\usepackage[open, openlevel=2, depth=3, atend]{bookmark}
\hypersetup{pdfstartview=XYZ}

\def\?[#1]{\textbf{[#1]}\marginpar{\Large{\textbf{??}}}}

\newtheorem{theo}{Theorem}
\newtheorem{prop}{Proposition}[section]
\newtheorem{defi}[prop]{Definition}

\newtheorem{lemm}[prop]{Lemma}
\newtheorem{corr}[prop]{Corollary}

\numberwithin{equation}{section}

\newcommand{\mc}{\mathcal}
\newcommand{\rr}{\mathbb{R}}
\newcommand{\nn}{\mathbb{N}}
\newcommand{\cc}{\mathbb{C}}

\newcommand{\zz}{\mathbb{Z}}

\newcommand{\la}{\lambda}
\newcommand{\eps}{\epsilon}

\newcommand{\pl}{\partial}
\newcommand{\x}{\times}

\newcommand{\bbar}{\overline}

\newcommand{\cjd}{\rangle}
\newcommand{\cjg}{\langle}

\newcommand{\demi}{\tfrac{1}{2}}

\DeclareMathOperator{\supp}{supp}

\def\indic{\operatorname{1\hskip-2.75pt\relax l}}

\title[Lens rigidity for manifolds with hyperbolic trapped set]{Lens rigidity for manifolds with hyperbolic trapped set}
\author{Colin Guillarmou}
\email{cguillar@dma.ens.fr}
\address{DMA, U.M.R. 8553 CNRS, \'Ecole Normale Superieure, 45 rue d'Ulm,
75230 Paris cedex 05, France}

\begin{document}
\maketitle

\begin{abstract}
For a Riemannian manifold $(M,g)$ with strictly convex boundary $\pl M$, 
the lens data consists in the set of lengths of geodesics $\gamma$ with endpoints on $\pl M$, 
together with their endpoints $(x_-,x_+)\in \pl M\x \pl M$ 
and tangent exit vectors $(v_-,v_+)\in T_{x_-} M\x T_{x_+} M$.
We show deformation lens rigidity for manifolds with hyperbolic trapped set and no conjugate points, 
a class which contains all manifolds with negative curvature and strictly convex boundary, including those with non-trivial topology and trapped geodesics. 
For the same class of manifolds in dimension $2$, we prove that the set  of endpoints and exit vectors of geodesics (ie. the scattering data) determines the Riemann surface up to conformal diffeomorphism.
\end{abstract}

\section{Introduction}

In this work, we study a geometric inverse problem concerning the 
recovery of a Riemannian manifold  $(M,g)$ with boundary 
from informations about its geodesic flow which can be read 
at the boundary.  Different aspects of this problem
have been extensively studied by \cite{Mu,Mi,Cr1,Ot, Sh,PeUh,StUh1,BuIv, CrHe}, 
among others. It also has applications to applied inverse problems, in geophysics and tomography.
Our results concern the case of negatively curved manifolds with convex
boundaries and more generally manifolds with hyperbolic trapped sets and no conjugate points. 
In those settings we resolve the deformation lens rigidity problem in all dimensions and in dimension $2$ we show that the lens data (and actually the scattering data) determine the Riemann surface up to conformal diffeomorphism. The difference with most previous
works is allowing trapping and non-trivial topology; we obtain the first general results 
in that case. With this aim in view, we introduce new methods making a systematic use of recent analytic  methods introduced in hyperbolic dynamical systems \cite{Li,FaSj, FaTs, DyZw,DyGu2}.

\subsection{Negative curvature}
Let $(M,g)$ be an $n$-dimensional oriented compact Riemannian manifolds with strictly convex boundary $\pl M$ (ie. the second fundamental form is positive). The incoming  (-) and outgoing (+) boundaries of the unit tangent bundle of $M$ are denoted
\[ \pl_\pm SM:=\{ (x,v)\in TM; x\in \pl M, |v|_{g_x}=1, \mp g_x(v,\nu)>0\}\] 
where $\nu$ is the inward pointing unit normal vector field to $\pl M$. 
For all  $(x,v)\in \pl_-SM$, the geodesic  
$\gamma_{(x,v)}$ with initial point $x$ and tangent vector $v$ has either infinite length or 
it exits $M$ at a boundary point $x'\in \pl M$ with tangent vector $v'$ with $(x',v')\in \pl_+SM$.
We call $\ell_g(x,v)\in [0,\infty]$ the length of this geodesic, and if
$\Gamma_-\subset \pl_- SM$ denotes the set of $(x,v)\in \pl_-SM$ with $\ell_g(x,v)=\infty$,
we call $S_g(x,v):=(x',v')\in \pl_+SM$ the exit pair or scattering image of $(x,v)$ when $(x,v)\notin \Gamma_-$. This defines the \emph{length map} and \emph{scattering map} 
\begin{equation} \label{Sell}
\ell_g: \pl_-SM\to [0,\infty], \quad S_g: \pl_-SM\setminus \Gamma_-\to \pl_+SM.
\end{equation}
and the lens data is the pair $(\ell_g,S_g)$. The lens data do not (a priori) contain information on closed geodesics of $M$, neither do they on geodesics not intersecting $\pl M$.

If $(M,g)$ and $(M',g')$ are two Riemannian manifolds with 
the same boundary $N$ and $g|_{TN}=g'|_{TN}$, there is 
a natural identification between $\pl_- SM$ and $\pl_- SM'$ since 
$\pl_-SM$ can be identified with the boundary ball bundle $BN:=\{(x,v)\in TN; |v|_{g}<1\}$ via the orthogonal projection  
$\pl SM\to BN$ with respect to $g$ (and similarly for $(M',g')$).
The \emph{lens rigidity problem} consists in showing that, if $(M,g)$ and $(M',g')$ are two Riemannian 
manifold metrics  
with strictly convex boundary and $\pl M=\pl M'$, then
\begin{equation}\label{question'}
\ell_{g}=\ell_{g'}, \, S_{g}=S_{g'} 
\Longrightarrow \exists \phi\in {\rm Diff}(M'; M), \,\, \phi^*g=g',\,\, \phi|_{\pl M'}={\rm Id}.
\end{equation}
When $(\ell_{g},S_{g})=(\ell_{g'},S_{g'})$, we say that $(M,g)$ and $(M', g')$ are 
\emph{lens equivalent}, while if $S_g=S_{g'}$ we say that they are \emph{scattering equivalent}.

Our first result is a deformation lens rigidity statement which holds in any dimension (this follows from Theorem \ref{Th1} below):
\begin{theo}\label{Th0bis}
For $s\in (-1,1)$, let $(M,g_s)$ be a smooth $1$-parameter family of metrics with negative curvature on a smooth $n$-dimensional manifold $M$ with strictly convex boundary, and assume that $g_s$ is lens equivalent to $g_0$ for all $s$, then there exists a family of diffeomorphisms $\phi_s$ which are 
equal to ${\rm Id}$ at $\pl M$ and with $\phi_s^*g_0=g_s$.  
\end{theo}
In dimension $2$, we show that the scattering data
determine the conformal structure (this is a corollary of Theorem \ref{Th2} below):
\begin{theo}\label{Th0}
Let $(M,g)$ and $(M',g')$ be two oriented negatively curved Riemannian surfaces with  strictly convex boundary such that  $\pl M=\pl M'$ and $g|_{T\pl M}=g'|_{T\pl M'}$. If $(M,g)$ and $(M',g')$ are scattering equivalent, then there is a diffeomorphism $\phi: M\to M'$ such that $\phi^*g'=e^{2\omega}g$ for some $\omega\in C^\infty(M)$ and $\phi|_{\pl M}={\rm Id}$, $\omega|_{\pl M}=0$.
\end{theo}
In the special case of simple manifolds, these results correspond to the much studied boundary rigidity problem, which consists in determining a metric (up to a diffeomorphism which is the identity on $\pl M$) on an $n$-dimensional Riemannian manifold $(M,g)$ with boundary $\pl M$ from the distance function $d_g:M\x M\to \rr$ restricted to $\pl M\x \pl M$.  A \emph{simple} manifold is a manifold with strictly convex boundary such that the exponential map  
$\exp_x: \exp_x^{-1}(M)\to M$ is a diffeomorphism at all points $x\in M$. 
Such manifolds have no conjugate points and no trapped geodesics (ie. geodesics entirely contained in $M^\circ:=M\setminus \pl M$), and between two boundary points $x,x'\in \pl M$ there is a unique geodesic in $M$ with endpoints $x,x'$. Boundary rigidity for simple metrics was conjectured by Michel \cite{Mi} and has been proved in some cases:\\
1) If $(M,g)$ and $(M,g')$ are conformal and lens equivalent simple manifolds, they are isometric; this is shown by Mukhometov-Romanov, Croke \cite{Mu,MuRo,Cr2}.\\
2) If $(M,g)$ and $(M',g')$ are lens equivalent simple surfaces ($n=2$), they are isometric.
This was proved by Otal \cite{Ot} in negative curvature and by Croke \cite{Cr1} in  non-positive curvature.
For general simple metrics, Pestov-Uhlmann \cite{PeUh} proved that the scattering data determine the conformal class and, combined with 1), this shows Michel's conjecture for $n=2$.\\ 
3) If $g$ and $g'$ are simple metrics that are close enough to a given simple analytic metric $g_0$, and are lens equivalent, then they are isometric. This was proved by Stefanov-Uhlmann \cite{StUh1}. 
All metrics $C^2$-close to a flat metric $g_0$ on a smooth domain of $\rr^n$ is boundary rigid, this was proved by Burago-Ivanov \cite{BuIv}.\\
4) A $1$-parameter smooth family of simple non-positive curved metrics with same lens data are all isometric, this was shown by Croke-Sharafutdinov \cite{CrSh}.

Thus, Theorem \ref{Th0} is similar to Pestov-Uhlmann's result in 2) for a class of non-simple surfaces and Theorem \ref{Th0bis} extends 4). We emphasize that in our case, there are typically infinitely many trapped geodesics (and closed geodesics) and this provides the first general rigidity result  in presence of trapping.
In fact, when there are trapped geodesics or when the flow has conjugate points, there exist lens equivalent metrics which are not isometric, see Croke \cite{Cr2} and Croke-Kleiner \cite{CrKl}.
So far, only results of lens rigidity in very particular cases were proved in case of trapped geodesics:\\ 
5)  Croke-Herreros \cite{CrHe} proved that a $2$-dimensional negatively curved or flat cylinder with convex boundary is lens rigid. Croke \cite{Cr3} showed that the flat product metric on $B_n\x S^1$ is scattering rigid if $B_n$ is the unit ball in $\rr^n$.\\
6) Stefanov-Uhlmann-Vasy \cite{SUV}  proved that the lens data near $\pl M$ determine the metric near $\pl M$ for metrics in a fixed conformal class, and more generally they recover the metric outside the convex core of $M$ under convex foliations assumptions.\\
7) For the flat metric on $\rr^n\setminus \mc{O}$ where $\mc{O}$ is a union of strictly convex domains, Noakes-Stoyanov \cite{NoSt} show that the lens data for the billiard flow on $\rr^n\setminus \mc{O}$ determine $\mc{O}$. \\

If $SM=\{(x,v)\in TM; |v|_{g_x}\}$ is the unit tangent bundle and $SM^\circ$ its interior, 
the \emph{trapped set} $K\subset SM^\circ$ of the geodesic flow is the set of 
points $(x,v)\in SM^\circ$ such that the geodesic passing through $x$ and tangent to $v$ does not intersect the boundary $\pl SM$; $K$ is a closed flow-invariant subset of $SM^\circ$ which includes all closed geodesics. 
In results 5) above, the trapped set has a simple structure, it is either two disjoint closed geodesics or an explicit smooth submanifold; in 6), it can be anything but the result allows only to
determine the metric near $\pl M$, which is the region of $M$ with no trapped geodesics. In comparison, in our case (in Theorem \ref{Th0bis} and \ref{Th0}), the trapped set is typically a complicated fractal set. For instance, in constant negative curvature they have Hausdorff dimension given in terms of the convergence exponent of the Poincar\'e series for the 
fundamental group (see \cite{Su}). 
  
\subsection{More general results} As mentioned above, the results obtained in negative 
curvature are particular cases of more general theorems. 
For $t\in\rr$, we denote by 
$\varphi_t$ the geodesic flow at time $t$ on $SM$, ie. $\varphi_t(x,v)=(x(t),v(t))$ where 
$x(t)$ is the point at distance $t$ on the geodesic generated by $(x,v)$ and $v(t)=\dot{x}(t)$ the tangent vector.
We say that the trapped set $K$ is a \emph{hyperbolic set} 
if there exists $C>0$ and $\nu>0$ so that for all $y=(x,v)\in K$, there is a continuous flow-invariant  splitting
\begin{equation}\label{hypdecomp}
T_y(SM)=\rr X(y)\oplus E_u(y)\oplus E_s(y)
\end{equation}
where $E_s(y)$ and $E_u(y)$ are vector subspaces satisfying 
\begin{equation}\label{hyperbolicK} 
\begin{gathered}
||d\varphi_t(y)w||\leq Ce^{-\nu t}||w||,\quad  \forall t>0, \forall w\in E_s(y),\\
||d\varphi_t(y)w||\leq Ce^{-\nu |t|}||w||,\quad  \forall t<0, \forall w\in E_u(y).
\end{gathered}
\end{equation} 
Here the norm is the Sasaki norm on $SM$ induced by $g$. This setting is quite natural and `interpolates' 
between the simple domain case (open, no trapped set) and the Anosov case (closed manifolds with hyperbolic geodesic flow). Negative curvature near the trapped set implies
that $K$ is a hyperbolic set, see \cite[\S3.9 and Theorem 3.2.17]{Kl2}, but although this is the typical example, 
negative curvature is a priori not necessary for that to happen. 
We show
\begin{theo}\label{Th2}
Let $(M,g)$ and $(M',g')$ be two oriented Riemannian surfaces with  strictly convex boundary such that  $\pl M=\pl M'$ and $g|_{T\pl M}=g'|_{T\pl M'}$. Assume that the trapped set of $g$ and $g'$ are hyperbolic and that the metrics have no conjugate points. If $(M,g)$ and $(M',g')$ are scattering equivalent, then there is a diffeomorphism $\phi: M\to M'$ such that $\phi^*g'=e^{2\omega}g$ for some $\omega\in C^\infty(M)$ and $\phi|_{\pl M}={\rm Id}$, $\omega|_{\pl M}=0$.
\end{theo}

In all dimension we obtain a deformation rigidity result:
\begin{theo}\label{Th1}
Let $M$ be a smooth compact manifold with boundary, equipped with a smooth $1$-parameter family 
of lens equivalent metrics $g_s$ for $s\in(-1,1)$ and assume that $\pl M$ is strictly convex for $g_s$ 
for each $s$. Suppose that, for all $s$, $g_s$ have hyperbolic trapped set.\\
1) If for all $s$, $g_s$ is conformal to $g_0$ and has no conjugate points, then $g_s=g_0$.\\
2) If $g_s$ has non-positive curvature, then there exists a family of diffeomorphisms $\phi_s$ which are 
equal to ${\rm Id}$ at $\pl M$ and with $\phi_s^*g_0=g_s$.  
\end{theo}

Theorem \ref{Th0} and \ref{Th0bis} follow from these results: 
 negatively curved metrics satisfy the assumptions of both theorems 
since these have no conjugate points.
Hyperbolicity of $K$ is a stable condition by small perturbations of the metric, 
and there is structural stability of hyperbolic sets for flows (see \cite[Chapter 18.2]{KaHa} and \cite{Ro}), 
which justifies the study of infinitesimal rigidity in that class of metrics. Other natural examples of such manifolds are strictly convex subset of closed manifold with Anosov geodesic flows.

\subsection{X-ray transform and Livsic type theorem}

One of the main tools for proving the results above is a precise analysis of the X-ray transform on tensors for manifolds with hyperbolic trapped set and no conjugate points. 
The $X$-ray transform of a function $f$ on $M$ is defined to be the  
set of integrals of $f$ along all possible geodesics with endpoints in $\pl M$, this is described by the operator
\[ I_0: C^\infty(M)\to C^\infty(\pl_-SM\setminus \Gamma_-), \quad I_0f(x,v)=\int_{0}^{\ell_g(x,v)}f(\pi_0(\varphi_t(x,v)))dt\]
where $\pi_0:SM\to M$ is the projection on the base.
We prove injectivity of $I_0$:
\begin{theo}\label{injofI_0}
Let $(M,g)$ be a Riemannian surface with  strictly convex boundary, hyperbolic trapped set and no conjugate points. Then for each $p>2$, the operator 
$I_0:L^p(M)\to L^2(\pl_-SM)$ is bounded and injective.
\end{theo}

We prove a similar theorem for the X-ray transform on $1$-forms, and when the curvature is non-positive, for $m$-symmetric tensors (see Theorem \ref{Inj} for a precise statement). 
We also obtain surjectivity of $I_0^*$ and prove that $I_0^*I_0$ is an elliptic pseudo-differential operator.
An important aspect of our analysis that is somehow surprising is that, even though the flow has trapped trajectories, the X-ray transform still fits into a Fredholm type problem like it does for simple domains. The main tool to show injectivity of $I_0$ 
is a Livsic theorem of a new type. Indeed, a H\"older Livsic theorem exists on the trapped set \cite[Th. 19.2.4]{ KaHa} but this is not very useful for our purpose.
The result we need and prove in Proposition \ref{kernelPi} is the following: if $f\in C^\infty(SM)$ integrates to $0$ along all geodesics relating boundary points of $M$, then there exists $u\in C^\infty(SM)$ satisfying $Xu=f$ and $u|_{\pl SM}=0$.
The method to prove this uses strongly the hyperbolicity 
of $K$, and a novelty here is that we make use of the theory of anisotropic Sobolev spaces adapted to the dynamic, which appeared recently in the field of hyperbolic dynamical systems (typically on Anosov flows \cite{BuLi,FaSj}) and  exponential decay of correlations \cite{Li}. To perform this analysis, we use microlocal tools developed recently in joint work with Dyatlov \cite{DyGu2} for Axiom A type dynamical systems. Another importance of this method 
is that it should give local uniqueness and stability estimates in any dimension for the boundary distance function in the universal cover (combining with methods of \cite{StUh1,StUh3}) and allow to deal with more general questions, like attenuated ray transform.

We also notice that a byproduct of Theorem \ref{injofI_0} (using  \cite[Th. 1.1]{DKLS}) is the existence of many new  examples with non-trivial topology and complicated trapped set where the Calder\'on problem can be solved in a conformal class.

\subsection{Comments} 
1) First, we notice that the assumption $g=g'$ on $T\pl M$ in Theorem \ref{Th2} is not a serious 
one and could be removed by standard arguments since, by \cite{LSU}, the length function near 
$\pl_0SM:=\{(x,v)\in \pl SM; \cjg\nu,v\cjd=0\}$ determines the metric on $T\pl M$ 
(we would then have to change slightly the definition of $S_g$, as in \cite{StUh3}).

2) A part of this work (in particular Section 4.3) deals with very general assumptions (no hyperbolicity assumption on $K$ and no assumptions on conjugate point) to describe solutions of the boundary value problems for transport equations in $SM$.

3) Contrary to the simple metric setting, the lens equivalence between two metrics does not induce a conjugation of geodesic flows, which makes the problem more difficult.

4) As pointed out to me by M. Salo, Theorem \ref{injofI_0} is sharp in the sense that if there exists a flat cylinder 
$\mc{C}=((-\eps,\eps)_\tau \x (\rr/a\zz)_\theta, d\tau^2+d\theta^2)$ (with $a>0$) embedded in a surface with strictly convex boundary,  then it is easy to check that $\ker I_0$ is infinite dimensional
and contains all functions $f$ compactly supported in $\mc{C}$, depending only on $\tau$ 
with $\int_{-\eps}^\eps f(\tau)d\tau=0$. In this case the trapped is of course not hyperbolic.

5) To prove Theorem \ref{Th2}, we show that the scattering map $S_g$ determines the space of boundary values of holomorphic functions on any surface with hyperbolic trapped set, no 
conjugate points. This result was first shown by Pestov-Uhlmann \cite{PeUh} in the case of simple domains. We use their commutator relation between flow and fiberwise Hilbert transform,
but we emphasize that due to trapping, several important aspects of their proof relating scattering map and boundary values of holomorphic functions are much more difficult to implement. 
To obtain the desired result, we need to address delicate questions which are absent in the non-trapping case: 
we need to solve boundary value problems for the transport equations in low regularity spaces and 
understand the wavefront set of solutions, we need to describe boundary values of invariant distributions in $SM$ with certain regularity only in terms of the scattering map $S_g$, we also need to prove injectivity of X-ray transform  on $1$-forms in certain negative Sobolev spaces. 
The use of the recent joint paper with Dyatlov \cite{DyGu2} is fundamental, and 
hyperbolicity of the flow on $K$ is very important to address these problems.
The space of boundary values of holomorphic functions allows to recover $(M,g)$ up to a conformal diffeomorphism by the result of Belishev \cite{Be}. We are not able to prove that the lens data determine the conformal factor. We think that it does but it is not an easy matter: indeed,  all proofs known in the simple domain case seem to fail in our setting due to the fact that 
there is an infinite set of geodesics between two given boundary points and the problem is that we do not know if the geodesics 
starting at $(x,v)\in \partial_-SM$ for lens equivalent conformal metrics  $g'=e^{2\omega}g$ and $g$ are homotopic. The difficulty of this question is related to the fact that small perturbations of the metric induce large perturbations for the geodesics passing through a fixed $(x,v)\in \pl_-SM$ if $\ell_g(x,v)$ is large, thus allowing for huge 
changes of the homotopy class to which the geodesic belongs. 

\textbf{Ackowledgements.} We thank particularly S. Dyatlov for the work \cite{DyGu2}, which is 
fundamentally used here. Thanks also to V. Baladi, S. Gou\"ezel, M. Mazzucchelli, F. Monard, V. Millot, F. Naud, G. Paternain, S. Tapie, G. Uhlmann, M. Zworski for useful discussions and comments. The research is partially supported by grants ANR-13-BS01-0007-01 and ANR-13-JS01-0006.

\section{Geometric setting and dynamical properties}
 
\subsection{Extension of $SM$ and the flow into a larger manifold}\label{subset:extension}
It is convenient to view $(M,g)$ as a strictly convex region of a larger smooth manifold $(\hat{M},\hat{g})$ 
with strictly convex boundary, and to extend the geodesic vector field $X$ on $SM$ 
into a vector field $X_0$ on $S\hat{M}$ which has complete flow, for instance by making $X_0$ vanish at 
$\pl S\hat{M}$. 

Let us describe this construction.
Near the boundary $\pl M$, let $(\rho,z)$ be normal coordinates to the boundary, 
ie. $\rho$ is the distance function to $\pl M$ satisfying $|d\rho|_g=1$ near $\pl M$ and $z$ are coordinates on $\pl M$. The metric then becomes 
$g=d\rho^2+h_\rho$ in a collar neighborhood $[0,\delta]_\rho \x \pl M$ of $\pl M$ 
for some smooth 1-parameter family $h_\rho$ of metrics on $\pl M$ and the strict convexity 
condition means that the second fundamental form $-\pl_\rho h_\rho|_{\rho=0}$ is a positive definite symmetric cotensor. 
We extend smoothly $h_\rho$ from $\rho\in [0,\delta]$ to $\rho\in [-1,\delta]$ 
as a family of metrics on $\pl M$ satisfying $-\pl_\rho h_\rho>0$ for all $\rho\in [-1,0]$. We 
can then view $M$ as a strictly convex region inside a larger manifold $M_e$ with strictly convex boundary as follows.  First, let 
$E=\pl M\x [-1,0]_\rho$ be the closed cylindrical manifold, and consider  
the connected sum $\hat{M}:=M\sqcup E$ where we glue the boundary $\{\rho=0\}\simeq \pl M$ of $E$ to the boundary 
$\pl M$ of $M$; then we put a smooth structure of manifold with boundary on $\hat{M}$ extending the smooth structure of $M$, we extend the metric $g$ smoothly from $M$ to $\hat{M}$ by setting $\hat{g}= d\rho^2+h_\rho$ in $E$.
Each hypersurface $\{\rho=c\}$ with $c\in [-1,0]$ is strictly convex.
We now set the extension
\[M_e:= \{y\in \hat{M}; \,\, y \in M\textrm{ or } y\in E \textrm{ and } \rho(y) \in [-\eps,0]\}\]
of $M$ for $\eps>0$ fixed small, so that $(M_e,g)$ is a manifold with strictly convex boundary containing $M$ and contained in $\hat{M}$. It is easily checked that the longest connected geodesic ray in $SM_e\setminus SM^\circ$ has length bounded by some $L<\infty$.
When $(M,g)$ has no conjugate point and hyperbolic trapped set, 
it is possible to choose $\eps$ small enough so that $(M_e,g)$ has no conjugate point either (see Section \ref{stable/unstable}), and we will do so each time we shall assume that $(M,g)$ has no conjugate point.
We denote by $X$ the geodesic vector field on 
the unit tangent bundle $S\hat{M}$ of $\hat{M}$ with respect to the extended metric $g$.
Let us define $\rho_0\in C^\infty(\hat{M})$ so that near $E$,  $\rho_0=F(\rho)$ is a smooth nondecreasing function of $\rho$ satisfying $F(\rho)=\rho+1$ near $\rho=-1$, and so that $\{\rho_0=1\}=M_e$. 
Denote by $\pi_0:S\hat{M}\to \hat{M}$ the projection on the base, then the rescaled vector field 
$X_0:= \pi_0^*(\rho_0) X$ 
on $S\hat{M}$ has the same integral curves as $X$, it is complete and $X_0=X$ in the neighborhood 
 $SM_e$ of $SM$. The flow at time $t$ of $X_0$ is denoted $\varphi_t$, and by strict convexity of 
 $M$ (resp. $M_e$) in $\hat{M}$, $\varphi_t$ is also the flow of $X$ in the sense that for all $y$ in $SM$ (resp. 
 in $SM_e$) one has $\pl_t\varphi_t(y)=X(\varphi_t(y))$ for $t\in [0,t_0]$ as long as $\varphi_{t_0}(y)\in SM$ (resp. 
 $\varphi_{t_0}(y)\in SM_e$). 
 
 We shall denote $M^\circ$ and $M_e^{\circ}$ for the interior of $M$ and  $M_e$.

\subsection{Incoming/outgoing tails and trapped set.} \label{incoming/outgoing}
We define the incoming (-), outgoing (+) and tangent (0) boundaries of $SM$ and $SM_e$
\[ \begin{gathered}
\pl_\mp SM:=\{ (x,v)\in \pl SM;  \pm d\rho(X)> 0\},\quad
\pl_\mp SM_e:=\{ (x,v)\in \pl SM_e;  \pm d\rho(X)> 0\},\\
\pl_0SM=\{ (x,v)\in \pl SM;  d\rho(X)= 0\}, \quad \pl_0SM_e=\{ (x,v)\in \pl SM;  d\rho(X)= 0\}.
\end{gathered}\]
For each point $(x,v)\in SM$, define the time of escape of $SM$ in positive (+) and negative (-) time:
\begin{equation}\label{ellpm} 
\begin{gathered}
\ell_+ (x,v)=\sup\, \{ t\geq 0; \varphi_{t}(x,v)\in SM\}\subset [0,+\infty],\\
\ell_- (x,v)=\inf\, \{ t\leq 0; \varphi_{t}(x,v)\in SM\}\subset [-\infty,0].
\end{gathered}
\end{equation}

\begin{defi}
The incoming (-) and outgoing (+) tail in $SM$ are defined by 
\[
\Gamma_\mp =\{(x,v)\in SM;  \ell_\pm(x,v)=\pm \infty\}=\bigcap_{t\geq 0}\varphi_{\mp t}(SM)
\]
and the trapped set for the flow on $SM$ is the set
\begin{equation}\label{trapped} 
K:=\Gamma_+\cap \Gamma_-=\bigcap_{t\in \rr}\varphi_t(SM).
\end{equation}
\end{defi}
We note that $\Gamma_\pm$ and $K$ are closed set and that $K$ is globally invariant by the flow. By the strict convexity of $\pl M$, the set $K$ is a compact subset of $SM^\circ$ since for all $(x,v)\in \pl SM$, 
 $\varphi_t(x,v)\in S\hat{M}\setminus SM$ for either all $t>0$ or all $t<0$.

Moreover, it is easy to check (\cite[Lemma 2.3]{DyGu2}) that $\Gamma_\pm$ are characterized by 
\begin{equation}\label{Gammapm} 
y\in \Gamma_\pm \iff d(\varphi_t(y),K)\to 0 \textrm{ as }t\to \mp \infty
\end{equation}
where $d(\cdot,\cdot)$ is the distance induced by the Sasaki metric. 
We then extend $\Gamma_\pm$ to $S\hat{M}$ by using the characterization \eqref{Gammapm}; the sets  
$\Gamma_\pm$ are closed flow invariant subsets of the interior $S\hat{M}^\circ$ of $S\hat{M}$. By strict convexity of the 
hypersurfaces $\{\rho=c\}$ with $c\in (-1,0]$, each point $y\in S\hat{M}$ with $\rho(y) \in (-1,0]$ is such that 
$d(\varphi_t(y),\pl S\hat{M})\to 0$ either as $t\to +\infty$ or $t\to -\infty$, and thus for all $c\in(0,1)$
\[K=\bigcap_{t\in\rr}\varphi_t(\{\rho_0 \geq c\})=\bigcap_{t\in\rr}\varphi_t(SM_e).\]
We also remark that the strict convexity of $\pl M$ and $\pl M_e$ implies
\begin{equation}\label{intersectGamma}
\Gamma_\mp\cap \pl SM=\Gamma_\mp\cap \pl_\mp SM, \quad 
\Gamma_\mp\cap \pl SM_e=\Gamma_\mp\cap \pl_\mp SM_e.
\end{equation}
Using the flow invariance of Liouville measure in $SM_e$, it is direct to check that 
(see the proof of Theorem 1  in \cite[Section 5.1]{DyGu1}) 
\begin{equation}\label{VolK} 
{\rm Vol}(K)=0\iff {\rm Vol}(SM_e\cap (\Gamma_-\cup \Gamma_+))=0.\end{equation}
where the volume is taken with respect to the Liouville measure.

The hyperbolicity of the trapped set $K$ is defined in the Introduction, and there is a flow-invariant 
continuous splitting of $T_{K}^*(SM)$ dual to \eqref{hypdecomp}, defined as follows: for all $y\in K$,
$T_y^*(SM)=E_0^*(y)\oplus E_s^*(y)\oplus E_u^*(y)$ where 
\[E_u^*(E_u\oplus \rr X)=0,\quad E_s^*(E_s\oplus \rr X)=0, \quad E_0^*(E_u\oplus E_s)=0.\]
We note that $E_0^*=\rr \alpha$ where $\alpha$ is the Liouville $1$-form.

\begin{figure}
\includegraphics[scale=0.5]{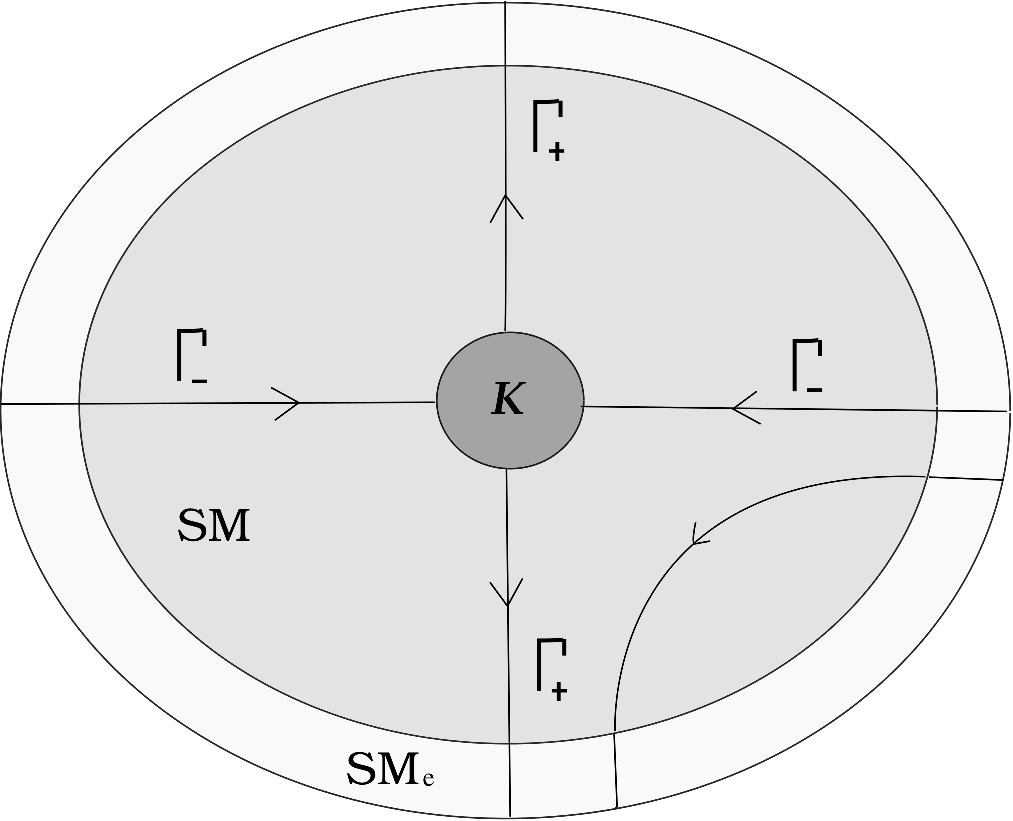}
\caption{The manifold $SM$ and $SM_e$}
\end{figure}

\subsection{Stable and unstable manifolds}\label{stable/unstable} Let us recall a few properties of flows with hyperbolic invariant sets, we refer to Hirsch-Palis-Pugh-Shub \cite[Sec 5 and 6]{HPPS}, Bowen-Ruelle \cite{BoRu} and Katok-Hasselblatt \cite[Chapters 17.4, 18.4]{KaHa} for details. 
For each point $y\in K$, there exist \emph{global stable and unstable manifolds} $W_s(y)$ and  $W_u(y)$ defined by 
\[\begin{gathered} 
W_s(y):=\{ y' \in S\hat{M}^\circ; d(\varphi_t(y),\varphi_t(y'))\to 0, t\to +\infty\},  \\ 
 W_u(y):=\{ y' \in S\hat{M}^\circ; d(\varphi_t(y),\varphi_t(y'))\to 0, t\to -\infty\}
\end{gathered}\] 
which are smooth injectively immersed connected manifolds. 
There are \emph{local stable/unstable manifolds} $W^\eps_s(y)\subset W_s(y)$, $W^\eps_u(y) 
\subset W_u(y)$ which are properly embedded disks containing $y$, defined by 
 \[ \begin{split}
 W^\eps_s(y):= & \{ y' \in W_s(y); \, \forall t\geq 0, \, d(\varphi_t(y),\varphi_t(y'))\leq \eps\},\\
  W^\eps_u(y):= & \{ y' \in W_u(y); \, \forall t\geq 0, \, d(\varphi_{-t}(y),\varphi_{-t}(y'))\leq \eps\}
 \end{split}\]
for some small $\eps>0$, 
\[ \begin{gathered}
 \varphi_t(W^\eps_s(y))\subset W^\eps_s(\varphi_t(y)) \textrm{ and }
\varphi_{-t}(W_u^\eps(y))\subset W^\eps_u(\varphi_{-t}(y)), \\
T_yW^\eps_s(y)=E_s(y), \textrm{ and }T_yW^\eps_u(y)=E_u(y), \\
\end{gathered}\]
The regularity of $W_u(y)$ and $W_s(y)$ with respect to $y$ is H\"older. We also define
\[\begin{gathered} 
W_s(K):=\cup_{y\in K}W_s(y), \quad W_u(K):=\cup_{y\in K}W_u(y), \\
W_s^\eps(K):=\cup_{y\in K}W_s^\eps(y), \quad W^\eps_u(K):=\cup_{y\in K}W^\eps_u(y). 
\end{gathered}\]
The incoming/outgoing tails are exactly the global stable/unstable manifolds of $K$:
\begin{lemm}\label{Gamma-W}
If the trapped set $K$ is hyperbolic, then the following equalities hold
\[ \Gamma_-=W_s(K), \quad \Gamma_+=W_u(K).\]
\end{lemm}
\begin{proof} By \eqref{Gammapm}, $W_s(K)\subset \Gamma_-$ and $W_u(K)\subset \Gamma_+$.
Then $W^\eps_s(K)\cap W_u^\eps(K)\subset K$,  and thus $K$ has a 
local product structure in the sense of \cite[Definition p.272 ]{KaHa}. 
Now from this local product structure,  \cite[Lemma 3.2 and Theorem 5.2]{HPPS} show that for any $\eps>0$ small, there is an open neighbourhood $V_K$ of $K$ such that 
\begin{equation}\label{VK} 
\{y\in SM_e; \varphi_t(y)\in V_K, \forall t\geq 0\} \subset W^\eps_s(K)
\end{equation}
which means that any trajectory which is close enough to $K$ is on the local stable manifold 
for $t$ large enough. The same hold for negative time and unstable manifold. A point $y\in \Gamma_-$ satisfies 
$d(\varphi_t(y),K)\to 0$ as $t\to +\infty$, thus for $t$ large enough the orbit reaches $V_K$ and thus 
$\varphi_t(y)\in W_s^\eps(K)$ for $t\gg 1$ large. We conclude that $y\in W_s(K)$. Similarly $\Gamma_+\subset W_u(K)$ and this achieves the proof.
\end{proof} 
For each  $y_0\in K$, we extend the notion of stable susbpace, resp. unstable subspace, to points on the $W^\eps_s(y_0)$ submanifold, resp. $W^\eps_u(y_0)$ submanifold, by 
\[ E_-(y):=T_yW^\eps_s(y_0) \textrm{ if }y\in W^\eps_s(y_0), 
\quad E_+(y):=T_yW^\eps_u(y_0) \textrm{ if }y\in W^\eps_u(y_0).\]
These subbundles can be extended to subbundles $E_\pm \subset T_{\Gamma_\pm}SM_e$ over $\Gamma_\pm$ in a flow invariant way (by using the flow), and we can define the subbundles 
$E_\pm^*\subset T_{\Gamma_\pm}^*SM_e$ by 
\begin{equation}\label{Epm^*}
 E_\pm^* (E_\pm\oplus \rr X)=0 \textrm{ over }\Gamma_\pm.
 \end{equation}
By \cite[Lemma 2.10]{DyGu2}, these subbundles are continuous, invariant by the flow and satisfy the following properties (we use Sasaki metric on $SM$):

\noindent 1) there exists $C>0,\gamma>0$ such that for all $y\in\Gamma_\pm$ and $\xi\in E_\pm^*(y)$, then 
\begin{equation}
  \label{estimE^*}
||d\varphi_t^{-1}(y)^{T}\xi||\leq  Ce^{-\gamma |t|}||\xi||, \quad \,\, \mp t>0
\end{equation}
2) for $(y,\xi)\in T^*_{\Gamma_\pm} SM_e$ such that $\xi\notin E_\pm^*$ and $\xi(X)=0$, then 
\begin{equation}\label{notinEpm^*}
||d\varphi_t^{-1}(y)^T\xi||\to \infty \,\,\textrm{ and } \,\, \frac{d\varphi_t^{-1}(y)^T\xi}{||d\varphi_t^{-1}(y)^T\xi||}\to 
E^*_\mp|_K \,\,\textrm{ as }\,\,t\to \mp ±\infty,
\end{equation}
3) The bundles $E_\pm^*$ extend $E_s^*$ and $E_u^*$ in the sense that $E_-^*|_K=E_s^*$ and $E_+^*|_K=E_u^*$.\\

The dependance of $E_\pm^*(y)$ with respect to $y$ is only H\"older.  The bundles $E_\pm^*$ can be thought of as 
conormal bundles to $\Gamma_\pm$ (this set is a union of smooth leaves parametrized by the 
set $K$ which has a fractal nature). 
The differential of the flow $d\varphi_t$ is exponentially contracting on each fiber $E_-(y)$, the proof of 
Klingenberg \cite[Proposition p.6]{Kl} shows 
\begin{equation}\label{klingenb}
\varphi_t \textrm{ has no conjugate points } \Longrightarrow  E_-\cap V=\{0\}\,\,   \textrm{if }V:=\ker d\pi_0
\end{equation}
where $\pi_0:SM\to M$ is the projection on the base. Similarly, $E_+\cap V=\{0\}$ in that case.
These properties imply
\begin{lemm}\label{extconj}
If $(M,g)$ has hyperbolic trapped set, strictly convex boundary, and no conjugate points, we can choose $\eps>0$ small enough in Section \ref{subset:extension}  so that the extension $(M_e,g)$ has not conjugate points. 
\end{lemm}
\begin{proof}
Indeed if it were not the case, there would be (by compactness) a sequence of points $(x_n,v_n)\in SM_e\setminus SM$ 
converging to $(x,v)\in \pl_-SM\cup \pl_0SM$ and $(x'_n,v'_n)\in SM_e$  converging to $(x',v')\in SM$, 
and geodesics $\gamma_n$ passing through $(x_n,v_n)$ and $(x'_n,v'_n)$, with $x_n$ and $x'_n$ being 
conjugate points for the flow of the extension of $g$. Note that $(x,v)=(x',v')$ is prevented by strict convexity of $\pl M$. 
By compactness, if the length of $\gamma_n$ is 
bounded, we deduce that $x,x'$ are conjugate points on $M$, 
which is not possible by assumption. There remains the case 
where the length of $\gamma_n$ is not bounded, we can take a subsequence so that the length $t_n\to +\infty$.
Then $(x,v)\in \Gamma_-$, and there is $w_n\in V=\ker d\pi_0$ of unit norm for Sasaki metric such that
$d\varphi_{t_n}(x_n,v_n).w_n\in V$. We can argue as in the proof of \cite[Lemma~2.11]{DyGu2}: 
by hyperbolicity of the flow on $K$,  for $n$ large enough, $d\varphi_{t_n}(x_n,v_n).w_n$ will be in an arbitrarily small 
conic neighborhood of $E_+$, thus it cannot be in the vertical bundle $V$. This completes the argument.
\end{proof}
Finally, let us denote by
\begin{equation}\label{iota} 
\iota_{\pm}: \pl_\pm SM\to SM_e, \quad \iota: \pl SM\to SM_e 
\end{equation} 
the inclusion map, and define 
\begin{equation}\label{Epmbound}
E_{\pl,\pm}^*:= (d\iota_{\pm})^TE_\pm^* \subset T^*(\pl_\pm SM).
\end{equation}

\subsection{Escape rate}\label{sec:escaperate}
An important quantity in the study of open dynamical systems is the \emph{escape rate}, which measures the 
amount of mass not escaping for long time. This quantity was studied for hyperbolic dynamical systems by 
Bowen-Ruelle, Young \cite{BoRu,Yo}. 
First we define the \emph{non-escaping mass function} $V(t)$ as follows
\begin{equation}\label{nonescape}
\begin{gathered}
V(t):={\rm Vol}(\mc{T}_+(t)), \textrm{ with } \\
\mc{T}_{\pm}(t):=\{ y\in SM;\, \varphi_{\pm s}(y)\in SM \textrm{ for } s\in[0,t]\}.
\end{gathered}\end{equation}
and ${\rm Vol}$ being the volume with respect to the Liouville 
measure $d\mu$.
The \emph{escape rate} $Q\leq 0$ measures the exponential rate of decay of $V(t)$ 
\begin{equation}
\begin{gathered}
\label{defQ}
Q:=\limsup_{t\to +\infty}\frac{1}{t}\log V(t).
\end{gathered}
\end{equation}
Notice that, since $\varphi_t$ preserves the Liouville measure in $SM$, we have 
\[{\rm Vol}(\mc{T}_+(t))=
{\rm Vol}(\mc{T}_-(t))\]
since the second set is the image of the first set by $\varphi_t$. Consequently, we also have
$Q=\limsup_{t\to +\infty}\frac{1}{t}\log {\rm Vol}(\mc{T}_-(t)).$
We define $J_u$ the unstable Jacobian of the flow   
\[J_u(y):=-\pl_{t}(\det d\varphi_t(y)|_{E_u(y)})|_{t=0}\]
where the determinant is defined using the Sasaki metric (to choose orthonormal bases 
in $E_u$).  
The \emph{topological pressure} of a continuous function $f:K\to \rr$ with respect to 
$\varphi_t$ can be defined by the variational formula  
$
P(f):=\sup_{\nu\in {\rm Inv}(K)}(h_{\nu}(\varphi_1)+\int_K f\, d\nu)
$
where ${\rm Inv}(K)$ is the set of $\varphi_t$-invariant Borel probability
measures and $h_\nu(\varphi_1)$ is the measure theoretic entropy of the flow
at time $1$ with respect to $\nu$ (e.g. $P(0)$ is just the
topological entropy of the flow).

We gather two results of Young \cite[Theorem~4]{Yo} and Bowen-Ruelle \cite[Theorem~5]{BoRu} on the escape rate in our setting.
\begin{prop}\label{youngbowenruelle}
If the trapped set $K$ is hyperbolic, the escape rate $Q$ is negative and given by the formula 
\begin{equation}
\label{youngformula}
Q=P(J_u).
\end{equation}
\end{prop}
\begin{proof}
Formula \eqref{youngformula} is proved by Young \cite[Theorem~4]{Yo} 
and follows directly from the volume lemma of Bowen-Ruelle \cite{BoRu}. 
The pressure $P(J_u)$ of  the unstable Jacobian  $J_u$ for $\varphi_1$ on $K$ is equal to the pressure 
$P(J_u|_{\Omega})$ of $J_u$ for $\varphi_1$ on the non-wandering set $\Omega\subset K$ 
of $\varphi_1$, see \cite[Corollary 9.10.1]{Wa}. By the spectral decomposition of hyperbolic flows \cite[Theorem 18.3.1 and Exercise 18.3.7]{KaHa}, the non-wandering set $\Omega$ decomposes into 
finitely many disjoint invariant topologically transitive sets $\Omega=\cup_{i=1}^N\Omega_i$ 
for $\varphi_1$.
By \cite[Corollary 6.4.20]{KaHa}, the periodic orbits of the flow are dense in $\Omega$.
By \cite[Proposition 7.2]{HPPS}, each component $\Omega_i$ of $\Omega$ has local product structure, and thus, according to \cite[Theorem 18.4.1]{KaHa} (see also \cite{HPPS}), it is locally maximal; each $\Omega_i$ is a basic set in the sense of Bowen-Ruelle  \cite{BoRu}. 

Then we can use the result of Bowen-Ruelle \cite[Theorem~5]{BoRu} which gives the following equivalence
\begin{equation}\label{bowenruelleth5} 
P(J_u|_{\Omega_i})<0\iff \Omega_i \textrm{ is not an attractor  for }\varphi_1 \iff {\rm Vol}(W_s(\Omega_i))=0.
\end{equation}
where  $W_s(\Omega_i):=\cup_{y\in \Omega_i}W_s(y)$ is the stable manifold of $\Omega_i$.
Suppose that one of the sets  $\Omega_i$ is an attractor, then $W_s(\Omega_i)$  has positive Liouville measure, implying that ${\rm Vol}(\Gamma_-)>0$, thus ${\rm Vol}(K)>0$ by \eqref{VolK}. 
Since Liouville measure is flow invariant on $K$, we have ${\rm Vol}(K)={\rm Vol}(\Omega)$ by \cite[Theorem 6.15]{Wa} and thus there is $\Omega_j$ with positive Liouville measure. Now we can conclude with the argument of \cite[Corollary 5.7]{BoRu}: ${\rm Vol}(W_s(\Omega_j))>0$ and 
${\rm Vol}(W_u(\Omega_j))>0$ so that $\Omega_j$ is an attractor for both $\varphi_1$ and $\varphi_{-1}$ by \eqref{bowenruelleth5}, and this implies that $W_u(\Omega_i)=\Omega_i$ (as an attractor of $\varphi_1$) 
and $W_u(\Omega_i)$ is open (as an attractor of $\varphi_{-1}$), thus $\Omega_j=K=SM$ since $SM$ is connected. But under our geometric assumption ($\pl M$ is strictly convex) this is not possible. 
We conclude that $Q=P(J^u|_{\Omega})<0$.
\end{proof}

This of course implies that ${\rm Vol}(\Gamma_-\cup \Gamma_+)=0$.
Near $\pl_\pm SM$, we have  $\{\varphi_{\mp t}(y)\in SM; t\in[0,\eps), y\in \pl_\pm SM\cap \Gamma_\pm\}\subset \Gamma_\pm$
and since for $U$ a small open neighborhood of $\pl_\pm SM\cap \Gamma_\pm$ the map
\[(t,y) \in [0,\eps)\x U\mapsto \varphi_{\mp t}(y)\in SM\] 
is a smooth diffeomorphism onto its image (the vector field $X$ is transverse to $\pl_\pm SM$ near $\Gamma_\pm$ by \eqref{intersectGamma}), 
we get 
\begin{equation}\label{VolGammapl}
{\rm Vol}_{\pl SM}(\Gamma_\pm\cap \pl_\pm SM)=0;
\end{equation} 
where the measure on $\pl SM$ is denoted $d\mu_{\pl SM}$ and given by $d\mu_{\pl SM}(x,v)=|{\rm dvol}_{h}(x)\wedge dS_x(v)|$ with $h=g|_{\pl M}$ and $dS_x(v)$ the volume form on the sphere $S_xM$.

The flow on $SM_e$ shares the same properties as on $SM$ and the trapped set on $SM$ and on $SM_e$ are the same, the discussion above holds as well for $SM_e$, and in particular \begin{equation}
\label{Qe+}
Q=\limsup_{t\to +\infty}\frac{1}{t}\log {\rm Vol}(\{ y\in SM_e;\, \varphi_{\pm s}(y)\in SM_e \textrm{ for }s\in[0, t]\})<
0.
\end{equation}
\subsection{Santalo formula}
There is a measure on $\pl SM$ which comes naturally when considering geodesic flow in $SM$, we denote 
it  $d\mu_\nu$ and it is given by 
\begin{equation}\label{dmunu}
d\mu_\nu(x,v):= |\cjg v,\nu\cjd|\, d\mu_{\pl SM}(x,v)= |\cjg v,\nu\cjd|\,|{\rm dvol}_{h}(x)\wedge dS_x(v)|.
\end{equation}
where $\nu$ is the inward unit normal vector field to $\pl M$ in $M$. This measures is also equal to 
$|\iota^*(i_X\mu)|$.
When ${\rm Vol}(\Gamma_-\cup \Gamma_+)=0$, then \eqref{VolGammapl} holds and we can apply Santalo formula \cite{Sa} to integrate functions in $SM$, this gives us: 
for all $f\in L^1(SM)$
\begin{equation}\label{santalo} 
\int_{SM}fd\mu = \int_{\pl_-SM\setminus \Gamma_-}\int_{0}^{\ell_+(x,v)}f(\varphi_t(x,v)) dt \,d\mu_\nu(x,v)
\end{equation}
with $\ell_+$ defined in \eqref{ellpm}. Extending $f$ to $SM_e$ by $0$ in $S\hat{M}\setminus SM$, \eqref{santalo} can also be rewritten
\begin{equation}\label{santalo2}
\int_{SM}fd\mu = \int_{\pl_-SM\setminus \Gamma_-}\int_{\rr}f(\varphi_t(x,v)) \, dt \,d\mu_\nu(x,v).
\end{equation}

\section{The scattering map and lens equivalence}
In the setting of a compact Riemannian manifold $(M,g)$ with strictly convex boundary $\pl M$, 
we define the \emph{scattering map} by
\begin{equation}\label{scattering} 
S_g :  \pl_-SM\setminus \Gamma_- \to \pl_+SM\setminus \Gamma_+, \quad S_g(x,v):=\varphi_{\ell_+(x,v)}(x,v)
\end{equation}
where $\ell_+(x,v)$ is the length of the geodesic $\pi_0(\cup_{t\in\rr}\varphi_t(x,v)))\cap M$, as defined in
\eqref{ellpm}. 
\begin{defi}
Let $(M_1,g_1)$ and $(M_2,g_2)$ be two Riemannian manifolds with the same boundary and such that $g_1=g_2$ on $T\pl M_1=T\pl M_2$ and the boundary is strictly convex for both metrics. Let $\nu_i$ be the inward pointing unit normal vector field on $\pl M_i$ and let $\Gamma_-^i\subset SM_i$ the 
incoming tail of the flow for $g_i$. Let $\alpha:\pl SM_1\to \pl SM_2$ be given by 
\begin{equation}\label{alpha}
\alpha(x,v+t\nu_1)=(x,v+t\nu_2) ,\quad \forall (v,t)\in T_x\pl M_1\x \rr, \,\, |v|_{g_1}^2+t^2=1.
\end{equation} 
Then $(M_1,g_1)$ and $(M_2,g_2)$ are said \emph{scattering equivalent} if 
\[\alpha(\Gamma_-^1)=\Gamma_-^2 , \textrm{ and } 
\alpha\circ S_{g_1}=S_{g_2}\circ \alpha \textrm{ on }\pl SM_1\setminus \Gamma_-^1.\]
Finally $g_1$ and $g_2$ are said \emph{lens equivalent} if they are scattering equivalent and for any $(x,v)\in \pl_-SM_1\setminus \Gamma_-^1$, 
the length $\ell^1_{+}(x,v)$ of the geodesic generated by $(x,v)$ in $M_1$ for $g_1$ 
is equal to the length $\ell^2_{+}(\alpha(x,v))$ of the geodesic generated by $\alpha(x,v)$ in $M_2$ for $g_2$. 
\end{defi} 

Let us show that for the case of surfaces, if $K$ is hyperbolic and $g$ has no conjugate points 
then $S_g$ determines the space $E_{\pl,\pm}^*$, this will be useful in Theorem \ref{scatrig}
\begin{lemm}\label{SgdetermE}
Let $(M,g)$ be a surface with strictly convex boundary. Assume that   
$K$ is hyperbolic and that the metric has no conjugate points. 
Then the scattering map $S_g$ determines  $E_{\pl, \pm}^*$.
\end{lemm}
\begin{proof}  All points in $\Gamma_+\cap \pl SM$ are in some unstable leaf $W_u(p)$ for some $p\in K$. 
The unstable leaves are one-dimensional manifolds injectively immersed in $SM_e$ 
and they intersect $\pl SM$ in a set of measure $0$ in $\pl SM$. Above a point $y\in W_u(p)\cap \pl_-SM$, the 
fiber $E_{+,\pl}^*(y)$ is exactly one-dimensional since one has 
$T_ySM=\rr X\oplus V\oplus E_+(y)$ where $V=\ker d\pi_0$ is the vertical bundle which is also tangent to $\pl SM$ 
and $E_-^*(V)\not=0$ if there are no conjugate points (we refer the reader to the proof of 
Proposition \ref{PsidoPi0} below for the discussion about that fact).
 Take a point $y\in W_u(p)\cap \pl_+SM$
and a sequence $y_n\to y$ in $\pl_+SM$ with $y_n\notin\Gamma_+$, then by compactness (by possibly passing to a subsequence) $z_{n}:=S_g^{-1}(y_{n})$ is converging to $z$ in $\Gamma_-\cap \pl SM$ with $t_n:=\ell_+(z_n)\to \infty$. 
We can write $S_g(z_n)=\varphi_{\ell_+(z_n)}(z_n)$.
By Lemma 2.11 in \cite{DyGu2} (in particular its proof), if $\xi_n\in T_{z_n}^*SM$ satisfies $\xi_n(X)=0$ and 
${\rm dist}(\xi_n/||\xi_n||, E_-^*)>\eps$ for some fixed $\eps>0$, then 
$(d\varphi_{t_n}(z_n)^{-1})^T\xi_n/ ||(d\varphi_{t_n}(z_n)^{-1})^T\xi_n||$ tends to $E_+(y)^*\cap S^*(SM)$.
Then we compute for $w_n\in T_{y_n}(\pl SM)$
\[ dS_g^{-1}(y_n).w_n= X(z_n)d\ell_-(y_n).w_n +d\varphi_{-t_n}(y_n).w_n\]
and if $\xi_n\in T_{z_n}^*(\pl SM)$, we can define uniquely $\xi_n^\sharp\in T_{z_n}^*SM$ by $\xi_n^\sharp(X)=0$ and 
$\xi^\sharp_n \circ d\iota=\xi_n$ ($\iota$ is defined in \ref{iota}) so that $(dS_g(z_n)^{-1})^T\xi_n= (d\varphi_{t_n}(z_n)^{-1})^T\xi^\sharp_n$.
We conclude that 
\[(dS_g(z_n)^{-1})^T\xi_n/||(dS_g(z_n)^{-1})^T\xi_n||\to E_+^*(y), \quad n\to +\infty\] 
if 
$\xi_n$ is such that ${\rm dist}(\xi^\sharp_n/||\xi^\sharp_n||, E_-^*)>\eps$. We can for instance take 
$\xi_n$ to be of norm $1$ and in the annulator of $V$ in $T^*\pl SM$, then the desired condition is 
satisfied and this shows that we can recover $E_-^*(y)$ from $S_g$. The same argument with $S_g^{-1}$ instead of $S_g$ shows that $S_g$ determines $E_+^*$. This ends the proof. 
\end{proof}

We can define the \emph{scattering operator} as the pull-back by the inverse scattering map 
\begin{equation}\label{scmap}
\mc{S}_g: C_c^\infty(\pl_-SM\setminus \Gamma_-)\to C_c^\infty(\pl_+SM\setminus \Gamma_+), \quad 
\mc{S}_g\omega_- =\omega_-\circ S_g^{-1}.
\end{equation}
\begin{lemm}\label{BVPsmooth}
For any $\omega_\mp \in C_c^\infty(\pl_\mp SM\setminus \Gamma_\mp)$, there exists a unique function 
$w\in C_c^\infty(SM\setminus (\Gamma_-\cup \Gamma_+))$  satisfying
\begin{equation}\label{invariantu} 
Xw=0 , \quad w|_{\pl_\mp SM}=\omega_\mp
\end{equation}
and this solution satisfies $w|_{\pl_+ SM}=\mc{S}_g\omega_-$ 
(resp. $w|_{\pl_- SM}=\mc{S}^{-1}_g\omega_+$). The function $w$ extends smoothly 
to $SM_e$ in a way that $Xw=0$, this defines a bounded operator
\begin{equation}\label{E_mp}
\mc{E}_\mp : C_c^\infty(\pl_\mp SM\setminus \Gamma_\mp)\to  C^\infty(SM_e),\quad \mc{E}_\mp(\omega_\mp):=w
\end{equation} 
which satisfies the identity $\mc{E}_+\mc{S}_g=\mc{E}_-$.
\end{lemm}
\begin{proof}
The function $w=\mc{E}_\mp(\omega_\mp)$ is simply given by 
\begin{equation}\label{expEmp}
w(x,v)=\mc{E}_\mp(\omega)(x,v)=\omega_\mp (\varphi_{\ell_\mp (x,v)}(x,v))
\end{equation}
in $SM$, and is clearly unique in $SM$ since constant on the flow lines. It is smooth in $SM$ 
since $\ell_{\pm}$ is smooth when restricted to $\pl_\pm SM\setminus \Gamma_\pm$, by the strict convexity of 
$\pl SM$. Then $\mc{E}_\mp(\omega_\mp)$ can be extended in $SM_e$ in a way that it is constant on the 
flow lines of $X$, satisfying $X\mc{E}_\mp(\omega_\mp)=0$. The continuity and linearity of $\mc{E}_\pm$ is obvious, and the identity $\mc{E}_+\mc{S}_g=\mc{E}_-$ comes from uniqueness of $w$.
Notice that $\supp(\mc{E}_\mp(\omega_\mp))$ is at positive distance from $\Gamma_-\cup\Gamma_+$ 
since $\omega_\mp$ has support not intersecting $\Gamma_\mp\cap \pl SM$.
\end{proof}
Denoting  $\omega_\pm:=\omega|_{\pl_\pm SM}$ if $\omega\in C_c^\infty(\pl SM\setminus (\Gamma_+\cup\Gamma_-))$, we now define the space 
\begin{equation}\label{CinftyS}
C^\infty_{S_g}(\pl SM):=\{ \omega\in C_c^\infty(\pl SM\setminus (\Gamma_+\cup\Gamma_-)); \,\,
\mc{S}_g\omega_-= \omega_+\}.
\end{equation}
Using the strict convexity and fold theory, Pestov-Uhlmann \cite[Lemma 1.1.]{PeUh} prove\footnote{Their result is for simple manifold, but the proof applies here without any problem since this is just an 
analysis near $\pl_0SM$ where the scattering map has the same behavior as on a simple manifold by the strict convexity of $\pl M$.} 
\begin{equation}\label{smoothext}
\omega \in C^\infty_{S_g}(\pl SM)\iff \exists w\in C_c^\infty(SM\setminus (\Gamma_-\cup \Gamma_+)),\,\,  Xw=0,\,\, w|_{\pl SM}=\omega.
\end{equation}
Similarly to \eqref{CinftyS}, we define the space
\begin{equation} \label{L2S}
L^2_{S_g}(\pl SM):=\{\omega \in L^2(\pl SM; d\mu_\nu);\,\,  \mc{S}_g\omega_-=\omega_+\}.
\end{equation}
We finally show
\begin{lemm}\label{unitary}
The map $\mc{S}_g$ extends as a unitary map 
\[L^2(\pl_-SM,d\mu_\nu) \to L^2(\pl_+SM,d\mu_\nu)\]
where $d\mu_\nu$ is the measure of \eqref{dmunu}.
\end{lemm}
\begin{proof}
Consider $w_-^1,w_-^2\in C_c^\infty(\pl_-SM\setminus \Gamma_-)$ and $w_1,w_2$ their invariant extension as in \eqref{invariantu}. Then we have 
\[ \begin{split}
0=&\int_{SM}Xw_1.\bbar{w_2}+w_1.X\bbar{w_2}\, d\mu= \int_{SM}X(w_1.\bbar{w_2})d\mu\\
  =& -\int_{\pl_-SM} \omega_-^1.\bbar{\omega_-^2}\,|\cjg X,N\cjd_S| \; d\mu_{\pl SM}+
         \int_{\pl_+SM} \mc{S}_g\omega_-^1.\,\bbar{\mc{S}_g\omega_-^2}\, |\cjg X,N\cjd_S|\, d\mu_{\pl SM}
\end{split}
\]
where $\cjg \cdot,\cdot\cjd_S$ is Sasaki metric and $N$ is the unit inward pointing normal vector field to $\pl SM$ for $S$. But $N$ is the horizontal lift of  $\nu$, and so 
$\cjg X,N\cjd_S=\cjg v,\nu\cjd_g$. This shows that $\mc{S}_g$ extends as an isometry by a density argument and reversing the role of $\pl_-SM$ with $\pl_+SM$ we see that $\mc{S}_g$ is invertible.
\end{proof}

\section{Resolvent and boundary value problem}

\subsection{Sobolev spaces and microlocal material} 
For a closed manifold $Y$, 
the $L^2$-based Sobolev space of order $s\in \rr$ is denoted $H^s(Y)$. If $Z$ is a manifold with a smooth boundary, it can be extended smoothly across its boundary as a subset of a closed manifold $Y$ of the same dimension; we denote by $H^s(Z)$ 
for $s\geq 0$ the $L^2$ functions on $Z$ which admit an $H^s$ extension to $Y$. The space $H_0^s(Z)$ is the closure of $C_c^\infty(Z^\circ)$ for the $H^s$ norm on $Y$ and we denote by $H^{-s}(Y)$ the dual of $H_0^s(Y)$.
We refer to Taylor \cite[Chap. 3-5]{Ta} for details and precise definitions. If $Z$ is an open manifold or a manifold with boundary, we 
set $C^{-\infty}(Z)$ to be the set of distributions, defined as the dual of $C_c^\infty(Z^\circ)$. For $\alpha\geq 0$, the Banach space $C^{\alpha}(Z)$ is the space of $\alpha$-H\"older functions. 
We will use the notion of wavefront set of a distribution (see \cite[Chap. 8]{Ho}), the calculus of pseudo-differential operators ($\Psi$DO in short), we refer 
the reader to Grigis-Sj\"ostrand \cite{GrSj} and Zworski \cite{Zw} for a thorough study. In particular, 
we shall say that a pseudo-differential operator $A$ on an open manifold $Z$ with dimension $n$
has support in $U\subset Z$ if its Schwartz kernel has support in $U\x U$.
The microsupport ${\rm WF}(A)$ (or wavefront set) of $A$ is defined as the complement to the 
set of points $(y_0,\xi_0)\in T^*Z$ such that there is a small neighborhood $U_{y_0}$ of $y_0$ 
and a cutoff function $\chi\in C_c^\infty(U_{y_0})$ equal to $1$ near $y_0$ such that $A_\chi:=\chi A\chi$ can be written under the form ($U_{y_0}$ is identified to an open set of $\rr^n$ using a chart)
\[ A_\chi f(y)=\int_{U_{y_0}}\int_{\rr^n}e^{i(y-y')\xi}\sigma(y,\xi)f(y')d\xi dy'\]
for some smooth symbol $\sigma$ satisfying $|\pl_{y}^\alpha \pl_\xi^\beta \sigma(y,\xi)|\leq C_{\alpha,\beta,N}\cjg \xi\cjd^{-N}$ for all $N>0$, $\alpha,\beta\in \nn^n$.

\subsection{Resolvent}\label{subsec:resolvent}
We first define the resolvent of the flow in the physical spectral region. 
\begin{lemm}\label{resolvphys} 
For ${\rm \la}>0$, the resolvents $R_\pm(\la): L^2(SM_e)\to L^2(SM_e)$
defined by the following formula 
\begin{equation}\label{R_+R_-}
\begin{gathered} 
R_+(\la)f(y)=\int_0^\infty e^{-\la t}f(\varphi_t(y))dt, \quad 
R_-(\la)f(y)=-\int_{-\infty}^0 e^{\la t}f(\varphi_t(y))dt
\end{gathered}
\end{equation}
are bounded.  They satisfy in the distribution sense in $SM_e^\circ$
\begin{equation}\label{inverse} 
\begin{gathered}
\forall f\in L^2(SM_e), \,\, (-X\pm \la)R_\pm(\la)f=f, \\ 
\forall f\in H_0^1(SM_e),\,\, R_\pm(\la)(-X\pm\la)f =f, 
\end{gathered}
\end{equation}
and we have the adjointness property
\begin{equation}\label{adjoint}
R_-(\bbar{\la})^*= - R_+(\la)  \textrm{ on }L^2(SM_e),
\end{equation} 
The expression \eqref{R_+R_-} gives an analytic continuation 
of $R_\pm(\la)$ to $\la\in\cc$ as operator
\begin{equation}\label{extensionla}
R_\pm(\la): C_c^\infty(SM_e^\circ \setminus \Gamma_\mp)\to C^\infty(SM_e)
\end{equation}
satisfying $(-X\pm\la)R_\pm(\la)f=f$ in $SM_e$, 
and an analytic continuation of $R_\pm(\la)\chi_\pm$ and  $\chi_\pm R_\pm(\la)$ as operators
\begin{equation}\label{Rchi}
\begin{gathered}
R_\pm(\la)\chi_\pm: L^2(SM_e)\to L^2(SM_e), \quad \chi_\pm R_\pm(\la): L^2(SM_e)\to L^2(SM_e)
\end{gathered}
\end{equation}
if $\chi_\pm \in C^\infty(SM_e)$ is supported in $SM_e\setminus \Gamma_\mp$. 
\end{lemm}

\begin{proof} The proof of \eqref{inverse} is straightforward. The boundedness on $L^2$  
follows from the inequality (using Cauchy-Schwarz)
\[ \int_{SM_e}\Big|\int_{0}^{\pm\infty} e^{-\la |t|}f(\varphi_t(x,v))dt\Big|^2d\mu\leq
C_\la\int_{SM_e}\int_{0}^{\pm\infty} e^{-{\rm Re}(\la) |t|}|f(\varphi_t(x,v))|^2dt d\mu, \]
for some $C_\la>0$ depending on ${\rm Re}(\la)$, and a change of variable $y=\varphi_t(x,v)$ with the fact that the flow $\varphi_t$ preserves the measure $d\mu$ in $SM_e$ gives the result.
The adjoint property \eqref{adjoint} is also a consequence of the invariance of $d\mu$ by the flow in  $SM_e$. 
The identity $(-X\pm \la)R_\pm(\la)f=f$ holds for any $f\in C_c^\infty(SM_e^\circ)$, thus for $f\in L^2(SM_e)$ 
and any $\psi\in C_c^\infty(SM_e^\circ)$ ($\cjg\cdot,\cdot\cjd$ is the distribution pairing)
\[\cjg (-X\pm \la)R_\pm(\la)f,\psi\cjd =\cjg R_\pm(\la)f,(X\pm\la)\psi\cjd=
\lim_{n\to \infty} \cjg R_\pm(\la)f_n,(X\pm\la)\psi\cjd=\lim_{n\to \infty}\cjg f_n,\psi\cjd
\]
if $f_n\to f$ in $L^2$ with $f_n\in C_c^\infty(SM_e)$, thus $(-X\pm\la)R_\pm(\la)f=f$ in $C^{-\infty}(SM_e^\circ)$.
The other identity in \eqref{inverse} is proved similarly.
The analytic continuation of $R_\pm(\la)$ in \eqref{extensionla} is direct to check by using
that the integrals in \eqref{R_+R_-} defining $R_\pm(\la)f$ are integrals on 
a compact set $t\in [-T,T]$ with $T$ depending  on the distance of 
support of $f$ to $\Gamma_\mp$. Similarly, the extension of
$R_\pm (\la)\chi_\pm f$ and $\chi_\pm R_\pm(\la)f$ for $f\in L^2(SM_e)$ comes from the fact that the 
support of $t\mapsto (\chi_\pm f)(\varphi_t(x,v))$ and of $t \mapsto \chi_\pm(x,v) f(\varphi_t(x,v))$ intersect $\rr^\pm$ in a compact set which is uniform with respect to $(x,v)\in SM_e$. 
\end{proof}

We next show that the resolvent at the parameter $\la=0$ can be defined if the non-escaping mass function 
$V(t)$ in \eqref{nonescape} is decaying enough as $t\to \infty$.
Let us first define the maximal Lyapunov exponent of the flow near $\Gamma_-\cup \Gamma_+$:
\begin{equation}\label{nu}
 \nu_{\max}=\max(\nu_+,\nu_-), \quad \textrm{ if } \nu_\pm:=\limsup_{t\to+ \infty}{1\over t}\log \sup_{(x,v)\in \mathcal T_\pm(t)}
\|d\varphi_{\pm t}(x,v)\|.
\end{equation}
where $\mc{T}_\pm$ is defined in \eqref{nonescape}.

\begin{prop}\label{boundedL2} 
Let $\alpha\in (0,1)$, let $Q$ be a negative real number and let $\nu_{\max}$ be the maximal Lyapunov exponent defined in \eqref{nu}.\\ 
1) The family of operators $R_\pm(\la)$ of Lemma \ref{resolvphys} extends as a continuous 
family in ${\rm Re}(\la)\geq 0$ of operators bounded on the spaces
\begin{equation}\label{boundLp1}
R_\pm(\la): L^\infty(SM_e)\to L^p(SM_e),  \,\, \textrm{ if }\int_{1}^\infty V(t)t^{p-1}dt<\infty \,\,\textrm{ with }  p\in [1,\infty),
\end{equation}
\begin{equation}\label{boundLp2}
R_\pm(\la):L^p(SM_e)\to L^1(SM_e), \,\, \textrm{ if }\int_{1}^\infty V(t)t^{\frac{1}{p-1}}dt<\infty \,\, \textrm{ with }   p\in (1,\infty),
\end{equation}
\begin{equation}\label{boundHs}
R_\pm(\la): C_c^\alpha(SM_e^\circ) \to H^s(SM_e), \,\, \textrm{ if } \,\,V(t)=\mc{O}(e^{Qt}) \textrm{ with } s<
\min\Big(\alpha,\frac{-Q}{2\nu_{\max}}\Big)
\end{equation} 
where $V(t)$ is the function of \eqref{nonescape}.
This operator satisfies $(-X\pm\la)R_\pm(\la)f=f$ in the distribution sense in $SM_e^\circ$ when $f$ is in one of the spaces where $R_\pm(\la)f$ is well-defined. \\
2) If $\iota: \pl SM\to SM_e$ is the inclusion map, then the 
operator $\iota^*R_\pm(\la)$ is a bounded operator on the spaces
\begin{equation}\label{restbound} 
L^\infty(SM_e)\to L^p(\pl SM), \quad L^p(SM_e)\to L^1(\pl SM), \quad 
C_c^\alpha(SM_e^\circ) \to H^s(\pl SM)\end{equation}
under the respective conditions \eqref{boundLp1}, \eqref{boundLp2} and \eqref{boundHs} on $V$, $p$ and $s$; the measure used on $\pl SM$ is $d\mu_\nu$, defined in \eqref{dmunu}.\\ 
3) If the condition \eqref{boundLp2} is satisfied and $f\in L^p(SM_e)$ has $\supp(f)\subset SM^\circ$, then 
$R_\pm(\la)f=0$ in a neighborhood of $\pl_\pm SM\cup \pl_0SM$ in $SM_e$.
\end{prop}
\begin{proof}
Let us denote $u_+(\la)=R_+(\la)f$ the $L^2$ function given by \eqref{R_+R_-} for ${\rm Re}(\la)>0$, 
when $f\in L^\infty(SM_e)$. 
When $\int_1^\infty V(t)t^{p-1}dt<\infty$, the measure of $\Gamma_+\cup \Gamma_-$ is $0$ and thus for $f\in L^\infty(SM_e)$ and $\la_0\in i\rr$, the function $u_+(\la_0; x,v):=\int_{0}^\infty e^{-\la_0t}f(\varphi_t(x,v))dt$ is 
finite outside a set of measure $0$ since $\ell_+^e(x,v)$, defined as the length of the geodesic 
$\{\varphi_t(x,v); t\geq 0\}\cap SM_e$, is finite on $SM_e\setminus \Gamma_-$.
If $\la_n$ is any sequence with ${\rm Re}(\la_n)>0$ converging to $\la_0$, we have $u_+(\la_n)\to u_+(\la_0)$ almost everywhere in $SM_e$ (using Lebesgue theorem). Moreover $|u_+(\la_n)|\leq \int_{0}^\infty |f\circ\varphi_t|dt$ almost everywhere 
in $SM_e$ for all $n>1$, and using Lebesgue theorem, we just need to prove
that $||R_+(0)(|f|) ||_{L^p}\leq C||f||_{L^\infty}$ to get that $||u_+(\la_0)||_{L^p}\leq C||f||_{L^\infty}$ and 
$u_+(\la_n)\to u_+(\la_0)$ in $L^p$.
We have for almost every $(x,v)$
\begin{equation}\label{u+exp} 
|u_+(0;x,v)|=\Big|\int_{0}^{\infty} f(\varphi_t(x,v))dt \Big| \leq ||f||_{L^\infty}\, \ell_+^e(x,v).
\end{equation}
Notice that, in view of our assumption on the metric in $SM_e\setminus SM$ we have 
$\ell_+(x,v)+L\geq \ell_+^e(x,v)\geq \ell_+(x,v)$ for some $L>0$ uniform in $(x,v)\in SM\setminus \Gamma_-$.
Using the definition of $V(t)$ in \eqref{nonescape},  the volume of the set $S_T$ of points $(x,v)\in SM_e$ such that 
$\ell_+^e(x,v)>T$ is smaller or equal to $2V(T-L)$ with $L$ as above (independent of $T$). We apply
Cavalieri principle for the function $\ell_+^e(x,v)$ in $SM_e\setminus \Gamma_-$, this gives 
\begin{equation}\label{boundu+} 
\int_{SM_e\setminus \Gamma_-} \ell^e_+(x,v)^p d\mu \leq 
C\Big(1+\int_{1}^{\infty}t^{p-1}V(t)dt\Big)
\end{equation}
which shows \eqref{boundLp1} using \eqref{u+exp}. Notice that the same argument gives the same bound
for the $L^p$ norms of $\ell_-^e$ in $SM_e\setminus \Gamma_+$.
The boundedness $L^p\to L^1$ of \eqref{boundLp2} is a direct consequence of 
\eqref{boundu+} (with $\ell_-$ instead of $\ell_+$) and the inequality 
\[ \int_{SM_e}\int_0^\infty |f(\varphi_t(x,v))|dtd\mu\leq \int_{SM_e}\int_0^\infty \indic_{SM_e}(\varphi_{-t}(x,v)) 
|f(x,v)|dtd\mu\leq ||\ell^e_-||_{L^{p'}} ||f||_{L^p}\]
for all $f\in C_c^\infty(SM_e^\circ)$ if $1/p'+1/p=1$.
The fact that $\iota^*R_\pm(\la)f$ defines a measurable function in $L^1(\pl SM,d\mu_\nu)$ when $f\in L^1(SM_e)$ comes directly from Santalo formula \eqref{santalo2} and Fubini theorem 
(note that $\pl_0SM$ has zero measure in $\pl SM$). 
This shows the boundedness property of $\iota^*R_\pm(\la): L^p(SM_e)\to L^1(\pl SM,d\mu_\nu)$.
Let us now prove the boundedness of the restriction $\iota^*R_\pm(0)f$ in $L^p$ when $f\in L^\infty$. 
Since $\ell^e_+(\varphi_t(x,v))=(\ell_+^e(x,v)-t)_+$ for $t>0$, Santalo formula gives  
\[\begin{gathered} 
\int_{\pl_- SM_e\setminus \Gamma_-} \int_0^{\ell_+^e(x,v)}\indic_{[T,\infty)}(\ell^e_+(x,v)-t)dt |\cjg v,\nu\cjd| d\mu_{\pl SM_e}={\rm Vol}(S_T) ,\\
  \int_{\pl_- SM\setminus \Gamma_-} \int_0^{\ell_+(x,v)}\indic_{[T,\infty)}(\ell_+(x,v)-t)dt  d\mu_\nu \leq {\rm Vol}(S_T) 
\end{gathered}\]
for $T$ large. From this, we get for large $T$ 
\begin{equation}\label{fonctionrepart}
\begin{gathered} 
\int_{\pl_- SM\setminus \Gamma_-} \indic_{[T,\infty)}(\ell_+(x,v)) d\mu_\nu \leq 2V(T-L-1),
\end{gathered}\end{equation}
and using Cavalieri principle, for any $\infty> p\geq 1$ there exists $C_p>0$ so that 
\begin{equation}\label{ell_+L^p}
\begin{gathered}
\int_{\pl_- SM\setminus \Gamma_-} \ell_+(x,v)^p d\mu_\nu \leq C\Big(1+\int_{1}^{\infty}
t^{p-1}V(t)dt\Big),
\end{gathered}\end{equation}
which shows, from \eqref{u+exp} that $u_+|_{\pl SM\setminus \Gamma_-}\in L^p$ for any $1\leq p<\infty$ with a bound $\mc{O}(||f||_{L^\infty})$.

To prove that $(-X\pm\la)R_\pm(\la)f=f$ in $C^{-\infty}(SM_e^\circ)$ when $f\in L^p$ for $p\in (1,\infty)$
and the condition $\int_{0}^\infty V(t)t^{p-1}dt<\infty$ is satisfied, we take 
$\psi\in C_c^\infty(SM^\circ_e\setminus \Gamma_\mp)$ and write 
\[ \cjg R_\pm(\la)f,(X\pm \la)\psi\cjd = \lim_{n\to \infty} \cjg R_\pm(\la)f_n,(X\pm \la)\psi\cjd 
=\lim_{n\to \infty}\cjg f_n,\psi\cjd=\cjg f,\psi\cjd \]  
where $f_n\in C_c^\infty(SM_e^\circ \setminus \Gamma_\mp)$ converges in $L^p$ to $f$; to obtain the second identity, we used \eqref{extensionla} and the fact that $(-X\pm \la)R_\pm (\la)f_n=f_n$ in $SM_e^\circ \setminus \Gamma_\mp$. 

Finally, we describe the case where the escape rate $Q$ is negative (ie. when $V(t)$ decays exponentially fast).
We need to prove that $u_+$ is in $H^s(SM_e)$ for some $s>0$ if $f\in C_c^\alpha(SM_e^\circ)$. 
To prove that $u_+$ is $H^s(SM_e)$, it suffices to prove (\cite[Chap. 7.9]{Ho}) 
\[ \int_{SM_e}\int_{SM_e}\frac{|u_+(y)-u_+(y')|^2}{d(y,y')^{n+2s}}dydy'<\infty\]
if $n=\dim(SM)$ and $d(y,y')$ denote the distance for the Sasaki metric on $SM_e$. Using that 
$f\in C^\alpha(SM_e)$, we have  that for all $\alpha\geq \beta>0$ small, there exists $C>0$ such that for all $y,y'\in SM_e$, $\nu>\nu_{\max}$ and all $t\in\rr$
\[ |f(\varphi_t(y))-f(\varphi_t(y'))|\leq C||f||_{C^\beta}e^{\nu \beta |t|}d(y,y')^{\beta}\]
thus for $\ell^e_+(y)<\infty$ and $\ell_+^e(y')<\infty$
\[ |u_+(y)-u_+(y')|\leq C\ell_+^e(y,y')e^{\nu \beta \ell_+^e(y,y')}d(y,y')^{\beta}.\]
where $\ell_+^e(y,y'):=\max(\ell_+^e(y),\ell_+^e(y'))$.
We then evaluate for $\beta-s>0$ and $\beta<\alpha$
\[\begin{split} 
\int \frac{|u_+(y)-u_+(y')|^2}{d(y,y')^{n+2s}}dydy'\leq &
 C_\beta \int e^{2\nu \beta \ell_+(y,y')}d(y,y')^{2(\beta-s)-n}dydy'\\
\leq & 2C_\beta\int_{\ell_+(y)>\ell_+(y')} e^{2\nu\beta \ell_+(y)}d(y,y')^{2(\beta-s)-n}dydy'\\
\leq & C_{s,\beta}\int_{SM_e} e^{2\nu \beta \ell_+(y)}dy
\end{split}
 \]
and from Cavalieri principle the last integral is finite if we choose $\beta>0$ small enough so that
$0<s<\beta<-Q/2\nu$. Taking $\nu$ arbitrarily close to $\nu_{\max}$ gives
that $u_+\in H^s(SM_e)$ if $s<-Q/2\nu_{\max}$.
The same argument works for $u_-$ and also for the boundary values $u_\pm|_{\pl SM}$.  

To finish, the proof of part 3) in the statement of the Proposition is a direct consequence of the expression 
\eqref{R_+R_-} for $R_\pm(\la)f$ since the positive (reps. negative) flowout  
of $\supp(f)\subset SM^\circ$ intersect $\pl SM$ in a compact region of $\pl_+SM$ (resp. $\pl_-SM$).
\end{proof}

\textbf{Remark.} Reasoning like in the proof Proposition \ref{boundedL2}, it is straightforward by using Cauchy-Schwarz to check that $R_\pm(\la)$ extends continuously to ${\rm Re}(\la)\geq 0$ as a family of bounded operators (restricting $R_\pm(\la)$ to functions supported on $SM$)
\[ R_\pm(\la) : \cjg \ell_\pm\cjd ^{-1/2-\eps}L^2(SM)\to \cjg \ell_\pm\cjd^{1/2+\eps}L^2(SM)\]
for all $\eps>0$ where $\ell_\pm$ is the escape time function of \eqref{ellpm}. This is comparable to the limiting absorption principle in scattering theory. The boundedness in Proposition \ref{boundedL2} are slightly finer and
describe the $L^p$ boundedness of $\ell_\pm$ in terms of $V(t)$ instead.

The resolvent $R_\pm(0)$ has been defined under decay property of the non-escaping mass function. In the case where $K$ is hyperbolic, we can actually say more refined properties of this operator. 
\begin{prop}[Dyatlov-Guillarmou \cite{DyGu2}]\label{DyGu}
Assume that the trapped set $K$ is hyperbolic. There exists $c>0$ such that for all $s>0$:\\ 
1) the resolvents $R_\mp (\la)$ extend meromorphically to the region 
${\rm Re}(\la)> -cs$ as a bounded operator 
\[R_\mp (\la) : H_0^s(SM_e)\to H^{-s}(SM_e)\]
with poles of finite multiplicity.\\
2) There is  a neighborhood $U_\mp$ of $E_\mp^*$ such that for all pseudo-differential operator $A_\mp$ of order $0$ with ${\rm WF}(A_\mp)\subset U_\mp$ and support  in $SM_e^\circ$, 
 $A_\mp R_\mp (\la)$ maps continuously $H_0^s(SM_e)$ to $H^{s}(SM_e)$, when $\lambda$ is not a pole.\\
3) Assume that $\la_0$ is not a pole of $R_\mp(\la)$, then the Schwartz kernel of $R_\mp (\la_0)$
is  a distribution on $SM^\circ_e\x SM^\circ_e$ with wavefront set 
\begin{equation}\label{WFRpm} 
{\rm WF}( R_\mp(\la_0))\subset N^*\Delta(SM^\circ_e\x SM^\circ_e)\cup \Omega_\pm\cup (E_\pm^*\x E_\mp^*).\end{equation} 
where $N^*\Delta(SM^\circ_e\x SM^\circ_e)$ is the conormal bundle to the diagonal 
$\Delta(SM^\circ_e\x SM^\circ_e)$ of $SM^\circ_e\x SM^\circ_e$ and 
\[\Omega_\pm:= \{(\varphi_{\pm t}(y),(d\varphi_{\pm t}(y)^{-1})^T\xi,y,-\xi)\in T^*(SM^\circ_e\x SM^\circ_e);  \,\, t\geq 0,\,\, \xi(X(y))=0\}.\] 
\end{prop}
\begin{proof}  Part 1) and 2) are stated in Proposition 6.1 of \cite{DyGu2},  (they actually follow from Lemma 4.3 and 4.4 of that paper), while part 3) is proved in Lemma 4.5 of \cite{DyGu2}. 
\end{proof}

We can now combine Propositions \ref{boundedL2} and \ref{DyGu} and obtain
\begin{prop}\label{Rpm0}
Assume that the trapped set $K$ is hyperbolic. Then we get for all $p<\infty$:\\
1) The resolvent $R_\pm(\la)$  
has no pole at $\la=0$, and it defines for all $s\in (0,1/2)$ a bounded operator $R_\pm(0)$ on the following spaces 
\[
R_\pm(0): H_0^s(SM_e)\to H^{-s}(SM_e) , \quad 
R_\pm(0): L^\infty(SM_e)\to L^p(SM_e) \] 
that satisfies  $-XR_\pm(0)f=f$ in the distribution sense, and for $f\in C^0(SM_e)$ one has
\begin{equation}\label{R0f}
\forall y\in SM\setminus \Gamma_\mp,\quad  (R_\pm(0)f)(y)=\int_0^{\pm \infty} f(\varphi_t(y))dt. 
\end{equation}
which is continuous in $SM\setminus \Gamma_\mp$ and satisfies $R_\pm(0)f|_{\pl_\pm SM}=0$ if 
$\supp(f)\subset SM$.\\
2) As a map $H_0^s(SM_e)\to H^{-s}(SM_e)$ for $s\in(0,1/2)$, we have 
\begin{equation}\label{adjoint0}
R_+(0)= -R_-(0)^*.
\end{equation}
3) If $f\in C_c^\infty(SM_e^\circ)$, the function $u_\pm:=R_\pm(0)f$ has wavefront set 
\begin{equation}\label{WFsetupm}
{\rm WF}(u_\pm)\subset E_\mp^*,
\end{equation}
the restriction $u_\pm|_{\pl SM}:=\iota^*u_\pm$ makes sense as a distribution satisfying 
\begin{equation}\label{WFsetupmpl}
 u_\pm|_{\pl SM}\in L^p(\pl SM), \quad {\rm WF}(u_\pm|_{\pl SM})\subset E_{\mp,\pl}^*.
 \end{equation}
4) Let $\alpha>0$, then for $f\in C^\alpha(SM)$ extended by $0$ on $SM_e\setminus SM$ as an element in $H_0^s(SM_e)$ for $s<\min(\alpha,1/2)$, we have $R_\pm(0)f\in H^s(SM_e)$ for $s<\min(\alpha,-Q/2\nu_{\max})$, where $\nu_{\max}$ is the maximal Lyapunov exponent  \eqref{nu}. Moreover 
$u_\pm|_{\pl SM}\in H^s(\pl_\pm SM)$ for such $s$.
\end{prop}
\begin{proof}
Recall that for ${\rm Re}(\la)>0$ we have for $f\in C_c^\infty(SM_e^\circ)$ and 
$\psi\in C_c^\infty(SM^\circ_e),$  
\[\cjg R_+(\la)f,\psi\cjd =\int_{0}^\infty e^{-\la t}\cjg f\circ \varphi_t,\psi\cjd \, dt.\] 
By Proposition \ref{boundedL2}, then as $\la\to 0$ along any complex half-line contained in 
${\rm Re}(\la)\geq 0$, we get $R_\pm(\la)f\to R_\pm(0)f$ in $L^p$ (thus in the distribution sense). This implies that the extended resolvent $R_\pm(\la)$ of Proposition \ref{DyGu} can not have poles at $\la=0$ by density of $C_c^\infty(SM_e^\circ)$ in any $H_0^s(SM_e)$. The same argument shows that $R_\pm(\la)$
is holomorphic in $\{{\rm Re}(\la)>Q\}$. The expression \eqref{R0f} comes from Proposition \ref{boundedL2}, which also implies the continuity of $R_\pm(0)f$ outside $\Gamma_\mp$ 
and its vanishing at $\pl_\pm SM$ when $\supp(f)\subset SM$.

Part 2) and \eqref{adjoint0} follows by continuity by taking $\la\to 0$ in \eqref{adjoint} 
(and applying on $H_0^s(SM_e)$ functions instead of $L^2(SM_e)$).

For part 3), the wavefront set property of $u_\pm:=R_\pm(0)f$ if $f\in C_c^\infty(SM_e^\circ)$ follows from the wavefront set description \eqref{WFRpm} of the Schwartz kernel of $R_\pm(0)$ and the composition rule of \cite[Theorem 8.2.13]{Ho}. The fact that $u_\pm$ restricts to $\pl SM$ as a distribution which satisfies \eqref{WFsetupmpl} comes from \cite[Theorem 8.2.4]{Ho} and the fact that $N^*(\pl SM)\cap E_\pm^*=0$, if $N^*(\pl SM)\subset T^*(SM_e)$ is the conormal bundle to $\pl SM$.
The $L^1(\pl SM)$ boundedness of the restriction follows from
\eqref{restbound}.

For part 4), the fact that the extension of $f$ by $0$ is in $H_0^s(SM_e)$ for $s\in (0,1/2)$ is proved in \cite[Proposition 5.3]{Ta}, and the rest is proved in Proposition \ref{boundedL2}.
\end{proof}

In fact, if $f\in C_c^\infty(SM_e^\circ)$  has support in $SM$, the expression \eqref{R0f} 
vanishes in a neighborhood of $\pl_+SM$ (resp. $\pl_-SM$) in $SM_e\setminus SM^\circ$, and thus 
\begin{equation}\label{Rpm0=0}
R_\pm(0)f\textrm{ vanishes to all order at }\pl_\pm SM.
\end{equation}
We also want to make the following observation: the involution $A: (x,v)\mapsto (x,-v)$ on $SM_e$ and $SM$
is a diffeomorphism and thus acts by pullback on distributions, it 
allows to decompose distributions $u$ on $SM_e^\circ$ into even and odd parts $u=u_{\rm ev}+u_{\rm od}$
where $u_{\rm od}:=\demi({\rm Id}-A^*)u$. If $f\in C_c^\infty(SM_e^\circ)$ is even, it is direct from 
 the expression \eqref{R0f} that 
\begin{equation}\label{evenodd}
(R_\pm(0)f)_{\rm ev}=\pm \demi (R_+(0)-R_-(0))f, \quad (R_\pm(0)f)_{\rm od}=\demi(R_+(0)+R_-(0))f
\end{equation}
and this extends by continuity to distributions. Similarly if $f$ is odd, $(R_\pm(0)f)_{\rm ev}=\demi(R_+(0)+R_-(0))f$
and $(R_\pm(0)f)_{\rm od}=\pm \demi (R_+(0)-R_-(0)) f$.

\subsection{Boundary value problem}
First, we extend the boundary value problem of Lemma \ref{BVPsmooth} to the case of 
$L^2(\pl_\mp SM)$ boundary data.
\begin{lemm}\label{BVPL2}
Assume that $\int_1^\infty tV(t)dt<\infty$ if $V$ is the function \eqref{nonescape}.
The map $\mc{E}_\mp$ of \eqref{E_mp} can be extended as a bounded operator 
$L^2(\pl_\mp SM,d\mu_\nu)\to  L^1(SM_e)$, satisfying $X\mc{E}_\mp(\omega_\mp)=0$ in the distribution sense for $\omega_\mp\in L^2(\pl_\mp SM,d\mu_\nu)$. 
\end{lemm}
\begin{proof} 
Using the expression \eqref{expEmp}, Santalo formula and Cauchy-Schwarz inequality, we see that there is $C>0$ such that for all $\omega_\mp \in C_c^\infty(\pl_\mp SM)$
\[ ||\mc{E}_\mp(\omega_\mp)||_{L^1(SM_e)}\leq C(||\omega_\mp||_{L^2(\pl_-SM,d\mu_\nu)}+||\ell_\pm^e||_{L^2(\pl_- SM,d\mu_\nu)}
||\omega_\mp||_{L^2(\pl_-SM,d\mu_\nu)})\]
where we used that there is $C'>0$ such that $|\ell_\mp^e(x,v)|\leq C'$ on $\pl_\mp SM$. Using 
\eqref{ell_+L^p}, we deduce the announced boundedness. The fact that $X\mc{E}_\mp=0$ on $L^2$ 
follows from the same identity on $C_c^\infty(\pl_\mp SM)$.
\end{proof}

In the case of a hyperbolic trapped set, using the resolvents $R_\pm(0)$, we are able to construct invariant distributions in $SM$ with prescribed value on $\pl_-SM$ and we can describe (partly) its singularities. 
\begin{prop}\label{boundval}
Assume that $K$ is hyperbolic, then:\\
1) For $\omega_- \in L^2(\pl_- SM,d\mu_\nu)$  
satisfying $\supp(\omega_-)\subset \pl_-SM$ and
\begin{equation}\label{WFcondition} 
{\rm WF}(\omega_-)\subset E_{\pl,-}^*,\,\, 
{\rm WF}(\mc{S}_g\omega_-)\subset E_{\pl,+}^*.
\end{equation}
the function $\mc{E}_-(\omega_-) \in L^1(SM_e)$ has wave-front set which satisfies 
\begin{equation}\label{WFEmp}
{\rm WF}(\mc{E}_-(\omega_-))\cap 
T^*(SM_e\setminus K)\subset E_-^*\cup E_+^*,
\end{equation}
the restriction $\mc{E}_-(\omega_-)|_{\pl_- SM}$ makes sense as a distribution in $L^2(\pl_- SM)$ 
and is equal to $\mc{E}_-(\omega_-)|_{\pl_- SM}=\omega_-$.\\ 
2) If $\omega_- \in H^s(\pl_- SM)$ for some $s>0$ with $\supp(\omega_-)\subset \pl_-SM$, 
if \eqref{WFcondition} holds and $\mc{S}_g\omega_-\in H^s(\pl_+SM)$, then $\mc{E}_-(\omega_-) \in H^s(SM_e)$. If $\pi_0:SM_e\to M_e$ is the projection on the base and ${\pi_0}_*$ the pushforward defined in \eqref{pushforward} then
\begin{equation}\label{moyennefibre}
{\pi_0}_*(\mc{E}_-(\omega_-))\in H_{{\rm loc}}^{s+\demi}(M_e).
\end{equation}
\end{prop}
\begin{proof} 
Let $U'\subset \pl_-SM$ be an open neighborhood of $\supp(\omega_-)$ whose closure 
does not intersect $\pl_0SM$ and let $U$ be 
the open neighborhood of $\supp(\omega_-)$ in $SM_e$ defined by 
$U=\cup_{-\infty <t<\eps}\varphi_t(U')\cap SM^\circ_e$ for some small $\eps>0$ so that $\bbar{U}$ 
does not intersect $\Gamma_+$. Then $U$ is diffeomorphic to an open subset $V$ of $(-\infty,\eps)\x U'$ by the map 
$\theta: (t,y)\mapsto \varphi_t(y)$. Assume that $\omega_-\in H^s(\pl_-SM)$ for some $s\in[0,1/2)$.
Using this parametrization, let $\psi_-\in H^s(U)$ be given by 
\[\psi_-(t,y):=\chi(t)\omega_-(y)\]
for some $\chi\in C^\infty(\rr)$ equal to $1$ near $\rr^-$ and equal to $0$ in $(\eps/2,\infty)$. 
Then, extend $\psi_-$ by $0$ in $SM_e\setminus U$, we still call it $\psi_-$. 
We first claim that ${\rm WF}(\psi_-)\subset E_-^*$. 
In the decomposition $V\subset (-\infty,\eps)\x U'$ of $U$ induced by the flow, the wavefront set of $\psi_-$ is 
$\{0\}\x E_{\pl,-}^*\subset T^*V$. The map $d\theta(0,y)^T$ maps 
the annulator of $\rr X_{y}$ to $\{0\}\x T^*U'$ and since 
$d\theta(0,y).(u,v)=uX(y)+d\iota(y).v$ where $\iota$ is the inclusion map, 
we have $(d\theta(0,y)^{-1})^TE_{\pl,-}^*(y)=E_-^*(y)$. And since the bundle 
$E_-^*$ is invariant by the flow,  $d\theta(t,y)^TE_-^*=\{0\}\x E_{\pl,-}^*$ thus we deduce that
${\rm WF}(\psi_-)\subset E_-^*$ and $\pi({\rm WF}(\psi_-))$ is at positive distance from 
$\Gamma_+$ if $\pi: T^*(SM_e)\to SM_e$ is the canonical projection. The restriction of $\psi_-$ on $\pl_-SM$ makes sense by \cite[Theorem 8.2.4]{Ho} since any element $\xi\in T^*(SM_e)$ conormal to $\pl SM$ 
and in ${\rm WF}(\psi_-)$ must satisfies $\xi(X)=0$ and $\xi|_{T(\pl SM)}=0$, thus $\xi=0$.
We also obviously have $\psi_-|_{\pl_-SM}=\omega_-$,
moreover $X\psi_-$ is supported in $SM^\circ$ and 
$X\psi_- \in H_0^s(SM_e)$. Using \eqref{boundLp2}, we obtain $R_-(0)(X\psi_-)\in L^1(SM_e)$, and by part 3) in Proposition \ref{boundedL2}, we also deduce that $R_-(0)(X\psi_-)=0$ in a neighborhood of 
$\pl_-SM\cup \pl_0SM$. Therefore, setting 
$w:= \psi_- -R_-(0)(X\psi_-)$, we have $w\in L^1(SM_e)$ and 
\[ Xw=0 \textrm{ in }SM_e \textrm{ and }w|_{\pl_-SM}=\omega_-.\]
Assume for the moment that $w\in C^\infty(SM_e\setminus (\Gamma_+\cup \Gamma_-))$ (we shall prove it below).
Then we claim that $w=\mc{E}_-(\omega_-)$ since both $w$ and $\mc{E}_-(\omega_-)$ are smooth flow 
invariant functions in $SM_e\setminus (\Gamma_+\cup \Gamma_-)$ agreeing on 
$\pl_-SM\setminus \Gamma_-$ and vanishing in the set $\{y\in SM_e; (\cup_{t\in \rr}\varphi_t(y)) \cap SM=\emptyset \}$.\\
Let us now prove \eqref{WFEmp}. Just as for the wave-front set analysis of $\psi_-$, 
${\rm WF}(X\psi_-)\subset E_-^*$ and $\pi({\rm WF}(X\psi_-))\subset SM^\circ$ is at positive distance from $\Gamma_+$. We recall the propagation of singularities for real principal type operator (see for instance 
\cite[Proposition 2.5]{DyZw}): 
let $\Phi_t:T^*(S\hat{M})\to T^*(S\hat{M})$ be the symplectic lift of $\varphi_t$, if $Xu=f$ then for each $T>0$
\begin{equation}\label{propag} 
\Phi_{\mp T}(y,\xi)\notin {\rm WF}(u),\,\, \bigcup_{t=0}^T\Phi_{\mp t}(y,\xi)\cap {\rm WF}(f)=\emptyset \Longrightarrow (y,\xi)\notin {\rm WF}(u).\end{equation}
Putting $u=R_-(0)(X\psi_-)$, we have $u=0$ near $\pl_-SM$ and thus all point
$(y,\xi)\notin E_-^*$ with $y\notin \Gamma_+$ is not in ${\rm WF}(u)$ by \eqref{propag}. This implies 
that 
\begin{equation}\label{WF1}
{\rm WF}(w)\cap T^*(SM_e\setminus \Gamma_+)\subset E_-^*
\end{equation}
and in particular $w$ is smooth in $SM_e\setminus (\Gamma_-\cup \Gamma_+)$, which implies that 
$\mc{E}_-(\omega_-)=w$, as mentioned above. By ellipticity and the equation 
$Xw=0$, we have 
\begin{equation}\label{WFwxiX}
{\rm WF}(w)\subset \{\xi\in T^*(SM_e); \xi(X)=0\}
\end{equation}
and $w$ smooth near $\pl_+SM\setminus \Gamma_+$, then as above we can use \cite[Theorem 8.2.4]{Ho} to deduce that 
the restriction $\omega_+:=w|_{\pl_+SM}$ makes sense as a distribution. 
Moreover it can be obtained as limits of restrictions $\mc{E}_-(\omega_{-}^{(n)})|_{\pl_+SM}$ 
where $\omega_-^{(n)}\in C_c^\infty(\pl_-SM)$ is a sequence converging in 
$L^2$ to $\omega_-$ (since also $\mc{E}_-(\omega_-^{(n)})$ has wave-front set contained 
in a uniform region not intersecting the conormal to $\pl_+SM$). Then, as 
$\mc{E}_-(\omega_-^{(n)})|_{\pl_+SM}=\mc{S}_g\omega_-^{(n)}$, we deduce from 
Lemma \ref{unitary} that $\mc{S}_g\omega_-=\omega_+$. By our assumptions on $\omega_-$,
we thus have $\omega_+\in H^{s}(\pl_+SM)$ and ${\rm WF}(\omega_+)\subset E_{\pl,+}^*$. Notice 
also that $\supp(\omega_+)\subset \pl_+SM$.
Then proceeding as above, but using the flow in backward direction, we can write 
$\mc{E}_-(\omega_-)=\mc{E}_+(\omega_+)=\psi_+-R_+(0)(X\psi_+)$ where 
$\psi_+\in H^{s}(SM_e)$ is defined similarly to $\psi_-$ but has support near 
$\supp(\omega_+)$ and ${\rm WF}(\psi_+)\subset E_+^*$. Then using similar arguments as above 
, ${\rm WF}(\mc{E}_-(\omega_-))\cap T^*(SM_e\setminus \Gamma_-)\subset E_+^*$
and combining with \eqref{WF1} this gives
\[{\rm WF}(\mc{E}_-(\omega_-))\cap T^*(SM_e\setminus K)\subset E_-^*\cup E_+^*.\]
Let us now prove that $w\in H^s(SM_e)$ if $s>0$. By point 2) in Proposition \ref{DyGu} applied to 
$R_\pm(0)(X\psi_\pm)$, we obtain 
that $A_\pm w\in H^s(SM_e)$ for some $s>0$ if $A_\pm$ is any $0$-th
order $\Psi$DO with ${\rm WF}(A_\pm)$ contained in a small enough neighborhood $V_\pm$ of 
$E_\pm^*$. Then if $B_1$  is any $0$-th
order $\Psi$DO with ${\rm WF}(B_1)$ contained outside 
an open neighborhood $V_1$ of $T_{K}^*(SM_e)$,  $B_1w\in H^s(SM_e)$.
By \eqref{WFwxiX}, we have $B_0w\in C^\infty(SM_e)$ if $B_0$ is any $0$-th
order $\Psi$DO with ${\rm WF}(B_0)$ contained outside a small conic neighborhood $V_0$
of the characteristic set $\{\xi\in T^*(SM_e); \xi(X)= 0\}$. Therefore, it
remains to prove that $B_2w\in H^s(SM_e)$ if 
$B_2$ is any $0$-th order $\Psi$DO with wave-front set contained in the region  
$V_2:=(V_0\cap V_1)\setminus (V_-\cup V_+)$. But this property will follow from propagation of singularities. Indeed, let $(y,\xi)\in V_2$, then the following alternative holds:\\ 
1) if $y\notin K$, there is $T>0$ such that either 
$\Phi_{T}(y,\xi)\notin V_1$ or $\Phi_{-T}(y,\xi)\notin V_1$\\
2) if $y\in K$, by \eqref{notinEpm^*} there is $T>0$ such that either 
$\Phi_{-T}(y,\xi)\in V_-$ or $\Phi_{T}(y,\xi)\in V_+$.\\  
We can apply \cite[Proposition 2.5]{DyZw} (recall that $Xw=0$), we obtain $B_2w\in H^s(SM_e)$ and this concludes the proof of $w\in H^s(SM_e)$.

To conclude, the $1/2$ gain in Sobolev regularity in \eqref{moyennefibre} follows from the averaging lemma of G\'erard-Golse \cite[Theorem 2.1]{GeGo}: indeed, the geodesic flow vector field, viewed as a first order differential operator satisfies the transversality assumption of Theorem 2.1 in \cite{GeGo} and thus, after extending slightly $w$ in an open neighborhood 
$W$ of $SM_e$ so that $Xw=0$ in $W$ and 
$w\in H^s(W)$, the averaging lemma implies that its average in the fibers 
${\pi_0}_*w$ restricts to $M_e$ as an $H_{\rm loc}^{s+1/2}$ function.
\end{proof}
Combining Proposition \ref{boundval} with \eqref{smoothext}, we obtain (using notation \eqref{L2S}) the following 
existence result for invariant distributions on $SM$ with prescribed boundary values. This will be 
fundamental for the resolution of the lens rigidity for surfaces.
\begin{corr}\label{corolbvp}
Assume that the trapped set $K$ is hyperbolic.
There exists an open neighborhood $U$ of $\pl SM\cap(\Gamma_-\cup\Gamma_+)$ in $SM_e^\circ$
such that for any $\omega\in L^2_{S_g}(\pl SM)$, 
satisfying ${\rm WF}(\omega)\subset E_{\pl,-}^*\cup E_{\pl,+}^*$, there exists  
$w\in L^1(SM_e)$ such that the restriction $w|_{\pl SM}$ makes sense as a distribution and
\[\begin{gathered} 
Xw= 0 \textrm{ in }SM\cup U, \quad w|_{\pl SM}=\omega, \\
{\rm WF}(w)\cap T^*(SM_e \setminus K) \subset E_-^*\cup E_+^*.
\end{gathered}\] 
If $\omega\in H^s(\pl SM)$ for $s>0$, then $w\in H^s(SM_e)$ and 
${\pi_0}_*w\in H_{\rm loc}^{s+1/2}(M_e)$. 
\end{corr}
\begin{proof} We decompose $\omega=\omega_1+\omega_2$ where $\omega_1\in C^\infty_{S_g}(\pl SM)$ with 
$\supp(\omega_1)\subset \pl SM\setminus (\Gamma_-\cup \Gamma_+)$ and $\omega_2$ supported near $\pl SM\cap (\Gamma_-\cup \Gamma_+)$. We apply \eqref{smoothext} to $\omega_1$, this produces 
$w_1\in  C^\infty(SM)$ which is flow invariant and with boundary value $\omega_1$. Then, we apply Proposition \ref{boundval}
to $\omega_2|_{\pl_-SM}$, this produces $w_2=\mc{E}_-(\omega_2|_{\pl_-SM})$ satisfying $Xw_2=0$ in $SM_e$ and $w_2|_{\pl_-SM}=\omega_2|_{\pl_-SM}$. 
Then set $w=w_1+w_2$. The wavefront set property of $w$ and the regularity of ${\pi_0}_*w$
follows from Proposition \ref{boundval}. 
\end{proof}

\section{$X$-ray transform and the operator $\Pi$}
We start by defining \emph{the $X$-ray transform} as the map
\[ I: C_c^\infty(SM\setminus \Gamma_-)\to C_c^\infty(\pl_-SM\setminus \Gamma_-), 
\quad If(x,v):=\int_{0}^\infty f(\varphi_t(x,v))dt.\]
From the expression \eqref{R0f}, we observe that
\begin{equation}\label{IvsR} 
If=(R_+(0)f)|_{\pl_-SM\setminus \Gamma_-}.
\end{equation}
Then $I$ can be extended to more general space. For instance, Santalo formula
implies directly that as long as ${\rm Vol}(K)=0$ (and no other assumption on $K$),
\[I: L^1(SM)\to L^1(\pl_-SM; d\mu_\nu).\]
For our purposes, as we shall see later, there is an important condition on the non-escaping mass function which allows to use $TT^*$ type arguments and relate $I^*I$ to the spectral measure at $0$ of the flow.  This condition is 
\begin{equation}\label{condescape}
\exists p\in (2,\infty], \quad \int_1^\infty t^{\frac{p}{p-2}}V(t)dt<\infty ,
\end{equation}
if $V$ is the function defined in \eqref{nonescape}. It is always satisfied if $K$ is hyperbolic.
We have 
\begin{lemm}\label{boundedI}
Assume that \eqref{condescape} holds for some $p>2$, then the X-ray transform $I$ extends boundedly as an operator
\[\begin{gathered}
I: L^p(SM)\to L^2(\pl_-SM,d\mu_\nu).
\end{gathered}\]
\end{lemm}
\begin{proof} Let $f\in L^p(SM)$, then using 
H\"older with $\tfrac{1}{p'}+\tfrac{1}{p}=1$ and $\tfrac{r}{p'}=\tfrac{p-1}{p-2}>1$, 
\[\begin{split} 
\int_{\pl_-SM}\Big|\int_{0}^{\ell_+(y)}f(\varphi_t(y))& dt\Big|^2d\mu_\nu(y)\leq 
\int_{\pl_-SM}\Big(\int_{0}^{\ell_+(y)}|f(\varphi_t(y))|^{p}dt\Big)^{2/p}\ell_+(y)^{2/p'}d\mu_\nu(y)\\
\leq &\Big(\int_{\pl_-SM}\int_{0}^{\ell_+(y)}|f(\varphi_t(y))|^{p}dtd\mu_\nu(y)\Big)^{2/p}
||\ell_+||_{L^{2r/p'}(\pl_-SM,d\mu_\nu)}^{2/p'}\\
 \leq & \, ||f||_{L^p(SM)}^{2}||\ell_+||_{L^{2r/p'}(\pl_-SM,d\mu_\nu)}^{2/p'}
\end{split}
\] 
where we have used Santalo formula to obtain the last line. Since $\ell_+\in L^q(\pl_- SM,d\mu_\nu)$ 
when $\int_1^\infty t^{q-1}V(t)dt$ by  \eqref{ell_+L^p}, we deduce the result. 
\end{proof}
Assume that $\int_1^\infty t^{p/(p-2)}V(t)dt<\infty$ for some $p\in (2,\infty)$. Note that by Sobolev embedding $ I:H_0^s(SM)\to L^2(\pl_-SM,d\mu_\nu)$ is bounded 
if $s=\tfrac{n}{2}-\tfrac{n}{p}$ for the $p\in (2,\infty)$ of Lemma \ref{boundedI}.
Since $H^{-s}(SM)$ is defined as the dual of $H_0^s(SM)$ and $L^{p'}$ is dual to $L^p$ for 
$p\in (2,\infty)$ if $1/p+1/p'=1$, the adjoint of $I$, denoted $I^*$, is bounded as operators (for $s$ as above)
\begin{equation}\label{I^*bound}
\begin{gathered} 
I^*:  L^2(\pl_-SM,d\mu_\nu)\to L^{p'}(SM), \quad 
I^*: L^2(\pl_-SM,d\mu_\nu)\to H^{-s}(SM).
\end{gathered}\end{equation}
In fact, a short computation gives 
\begin{lemm}\label{I^*}
If \eqref{condescape} holds true, then $I^*=\mc{E}_-$.
\end{lemm}
\begin{proof} Let $\omega_-\in C_c^\infty(\pl_-SM\setminus \Gamma_-)$, then 
$\mc{E}_-(\omega_-)\in C^\infty(SM)$ and its support does not intersect $\Gamma_-\cup \Gamma_+$.
By Green's formula, we have for $f\in C_c^\infty(SM^\circ)$ 
\[ \int_{SM} f\mc{E}_-(\omega_-)d\mu= \int_{SM} -X(R_+(0)f).\mc{E}_-(\omega_-)d\mu=
\int_{\pl_-SM}If.\omega_- |\cjg X,N\cjd|d\mu_{\pl SM}\]
where $S$ is Sasaki metric and $N$ the inward pointing unit normal to $\pl SM$ in $SM$. Like in the proof of Lemma \ref{unitary}, $|\cjg X,N\cjd_S|=|\cjg v,\nu\cjd|$. Using density of $C_c^\infty(SM^\circ)$ in $L^p(SM)$
and of $C_c^\infty(\pl_-SM\setminus \Gamma_-)$ in $L^2(\pl_-SM,d\mu_\nu)$, we get the desired result.
\end{proof}

To describe the properties of $I$ and $I^*$, it is convenient to define the operator 
\begin{equation}\label{operatorPi}
\Pi:= I^*I: L^p(SM)\to L^{p'}(SM), \quad \textrm{ when }\,\,\int_1^\infty t^{\frac{p}{p-2}}V(t)dt<\infty. \,\,
\end{equation}
for $p\in (2,\infty)$.
We prove the following relation between $\Pi$ and the resolvents:
\begin{lemm}\label{PivsR}
Assuming \eqref{condescape}, the operator $\Pi=I^*I$ of \eqref{operatorPi} is equal on $L^p(SM)$ to
\[\Pi= R_+(0)-R_-(0) \] 
\end{lemm}
\begin{proof} Since $\cjg R_+(0)f,f\cjd=-\cjg f,R_-(0)f\cjd$ by \eqref{adjoint0}, it suffices to prove the identity  
\[\cjg I^*If,f\cjd_{L^2(\pl_-SM,d\mu_\nu)}=
2\cjg R_+(0)f,f\cjd\] 
for all $f\in C_c^\infty(SM\setminus (\Gamma_-\cup \Gamma_+))$ real valued.  We write $u=R_+(0)f$ and compute, using Green's formula, 
\[\begin{split}
\int_{SM}u.f d\mu=  -\int_{SM}u.Xu d\mu = -\demi \int_{SM}X(u^2)d\mu=
\demi \int_{\pl_-SM}u^2|\cjg v,\nu\cjd|d\mu_{\pl SM}
\end{split}\] 
and this achieves the proof.
\end{proof}
With the assumption of Lemma \ref{PivsR}, the operator $\Pi$ can also be extended as a bounded operator $\Pi^e$ on $SM_e$ 
\begin{equation}\label{defPi}
\Pi^e:= R_+(0)-R_-(0): L^p(SM_e)\to L^1(SM_e), 
\end{equation} 
satisfying $\Pi^ef|_{SM}=\Pi f$ for all $f\in L^p(SM)$ extended by $0$ on $SM_e\setminus SM$. 
As above, one directly sees that $\Pi^e={I^e}^*I^e$ if we call $I^e: L^p(SM_e)\to 
L^2(\pl_-SM_e; |\cjg v,\nu\cjd|d\mu_{\pl SM_e})$ the X-ray transform on $SM_e$, defined just as on $SM$ and satisfying the same properties. In particular this shows that $\Pi^e: L^p(SM_e)\to L^{p'}(SM_e)$ is bounded.
We summarize the discussion by the following:
\begin{prop}\label{Pi} 
Assume that \eqref{condescape} holds for $p\in (2,\infty)$. Then we obtain \\
1) the operator $\Pi^e$ is bounded and self-adjoint as a map 
\[ \Pi^e: L^p(SM_e)\to L^{p'}(SM_e),  \quad 1/p+1/p'=1,\]
it satisfies for each $f\in L^p(SM_e)$
\begin{equation}\label{Xpi} 
X\Pi^ef =0
\end{equation} 
in the distribution sense and $\Pi^ef$ is given, outside a set of measure $0$, by the formula  
\begin{equation}\label{formulePi} 
\Pi^e f(x,v)= \int_{-\infty}^\infty f(\varphi_t(x,v))dt.
\end{equation}
2) If the trapped set $K$ is hyperbolic, the operator $\Pi^e: H_0^s(SM_e)\to H^{-s}(SM_e)$
is bounded for all $s\in(0,1/2)$. For each $f\in C_c^\infty(SM_e^\circ)$, the expression \eqref{formulePi} holds in $SM_e\setminus (\Gamma_+\cup \Gamma_-)$, we have 
 ${\rm WF}(\Pi^e f)\in E_-^*\cup E_+^*$, the restriction $\omega_\pm:=(\Pi^e f)|_{\pl_\pm SM}$ makes sense as a distribution, is in $L^2(\pl_\pm SM,d\mu_\nu)$ with wave-front set
\begin{equation}\label{wboundary} 
{\rm WF}(\omega_\pm)\subset E_{\pl,\pm}^*,
\end{equation}
and $\mc{S}_g\omega_-=\omega_+$ where $\mc{S}_g$ is the scattering map \eqref{scmap}. Finally 
$\omega_{\pm}\in H^s(\pl_\pm SM)$ 
for all $s<-Q/2\nu_{\max}$ with $\nu_{\max}$ defined in \eqref{nu}.
\end{prop}
\begin{proof} The boundedness and the self-adjoint property have already been proved. The 
property \eqref{Xpi} is clear from the properties of $R_\pm(0)$ given in 1) of Proposition \ref{Rpm0}.
The expression of $\Pi^e f$
follows from \eqref{R0f} (and the proof of Proposition \ref{boundedL2} for the extension to $L^p$ functions).
The wavefront set property of $w$ follows from \eqref{WFsetupm}, and the wavefront set and regularity properties \eqref{wboundary} of the restrictions $\omega_\pm$ are consequences of Proposition \ref{Rpm0}. The fact that $\omega_-\in L^2(\pl_-SM,d\mu_\nu)$ comes from Lemma \ref{boundedI}.
Finally, $\mc{S}_g\omega_-=\omega_+$ since, by \eqref{formulePi},
$\omega_+=\omega_-\circ S_g^{-1}$ on $\pl_+SM\setminus\Gamma_+$, this also implies that 
$\omega_+\in L^2(\pl SM,d\mu_\nu)$ by Lemma \ref{unitary}.  The last statement in the Proposition is a consequence of 3) in Proposition \ref{boundedL2}.  
\end{proof}
Next, we describe the kernel of $\Pi^e$ restricted to smooth functions supported in $SM$.
\begin{prop}\label{kernelPi}
Assume that $K$ is hyperbolic. Let $f\in C^\infty(SM)$ extended by $0$ in $SM_e\setminus SM$, if 
$\Pi^e f=0$ in $SM$, there exists $u\in C^\infty(SM)$ vanishing at $\pl SM$ such that $Xu=f$. If $f$ vanishes to infinite 
order at $\pl M$, then $u$ also does so.
\end{prop}
\begin{proof}
First, the extension of $f$ by $0$ can be viewed as an element in $H_0^s(SM_e)$ for $s<1/2$ with ${\rm WF}(f)\subset 
N^*(\pl SM)$ 
where $N^*(\pl SM)$ is the conormal bundle of $\pl SM$ in $SM_e$. By the composition law of wave-front set in 
\cite[Theorem 8.2.13]{Ho} and \eqref{WFRpm}, we deduce that 
\[\begin{gathered} 
{\rm WF}(R_\mp(0)f)\subset N^*\pl SM\cup E_\pm^*\cup B_\mp \\
B_\pm:= \cup_{t\geq 0}\{(\varphi_{\pm t}(y),(d\varphi_{\pm t}(y)^{-1})^T\xi)
\in T^*SM^\circ_e; y\in \pl_0SM, \xi\in N^*(\pl SM)\}
\end{gathered}\]
Clearly, by strict convexity, $B_\pm$ projects down to $M_e\setminus M^\circ$.
Now, the function $\ell_{\pm}$ is smooth in $SM\setminus (\pl_0SM\cup \Gamma_-\cup \Gamma_+)$
and from the expression \eqref{R0f} and the smoothness of $f$, we then get that 
$R_\mp(0)f$ is smooth in $SM\setminus (\pl_0SM \cup \Gamma_\pm)$ and $(R_\pm(0)f)|_{\pl_\pm SM}=0$.
To analyze the regularity at $\pl_0SM$, we decompose $f=f_{\rm ev}+f_{\rm od}$,  
we get by \eqref{evenodd} that $(R_\pm(0)f_{\rm ev})_{\rm ev}=\pm \demi \Pi^ef=0$ 
and similarly $(R_\pm(0)f_{\rm od})_{\rm od}=0$. Now  the argument of 
\cite[Lemma 2.3]{SaUh} shows that $(R_\pm(0)f_{\rm ev})_{\rm od}|_{SM}$ 
and $(R_\pm(0)f_{\rm od})_{\rm ev}|_{SM}$ are both smooth near $\pl_0SM$, which implies 
that $R_\pm(0)f$ is smooth near $\pl_0SM$ in $SM$.
Since $R_+(0)f=R_-(0)f$ if $\Pi^ef=0$, we deduce 
that $(R_\pm(0)f)|_{SM}\in C^\infty(SM\setminus K)$ and $(R_\pm(0)f)|_{\pl SM}=0$. From the wavefront set description above and the fact that 
$E_+^*\cap E_-^*=\{0\}$ over $K$, we conclude that $(R_\mp(0)f)|_{SM}\in C^\infty(SM)$. It just suffices to set 
$u=R_+(0)f$ to conclude the proof.
The fact that $f$ vanishes to all order at $\pl SM$ implies that $R_\pm(0)f$ 
vanishes to all order at $\pl_\pm SM$ by \eqref{Rpm0=0}, and thus $u$ vanishes to all order at $\pl SM$.
\end{proof}

\subsection{The operators $I_0$ and $\Pi_0$}

Here we deal with the analysis of X-ray transform acting on functions on $M$. The projection  
$\pi_0:SM_e\to M_e$ on the base induces a pull-back map 
\[ \pi_0^* : C_c^{\infty}(M_e^\circ)\to C_c^{\infty}(SM^\circ_e), \quad \pi_0^*f:=f\circ \pi_0\]
and a push-forward map ${\pi_0}_*$ defined by duality
 \begin{equation}\label{pushforward}
 {\pi_0}_*: C^{-\infty}(SM^\circ_e)\to C^{-\infty}(M^\circ_e), \quad \cjg {\pi_0}_*u,f\cjd:= 
\cjg u,\pi_0^*f\cjd.\end{equation}
Push-forward corresponds to integration in the fibers of $SM_e$ when acting on smooth functions.
The pull-back by $\pi_0$ also makes sense on $M$ and gives a bounded operator 
$\pi_0^*: L^p(M)\to L^p(SM)$ for all $p\in(1,\infty)$.
When \eqref{condescape} holds for some $p\in (2,\infty)$, we
define the X-ray transform on functions as the bounded operator (see Lemma \eqref{boundedI}) 
\begin{equation}\label{defI0}
I_0 := I\, \pi_0^*: L^p(M)\to L^2(\pl_-SM, d\mu_\nu).
\end{equation} 
The adjoint $I_0^*:L^2(\pl_-SM, d\mu_\nu)\to L^{p'}(M)$ is bounded if 
$1/p'+1/p=1$ and it is given by $I_0^*={\pi_0}_*I^*$. The operator $\Pi_0$ is simply defined as the bounded
self-adjoint operator for $p\in (2,\infty)$ and $1/p'+1/p=1$ 
\begin{equation}\label{defofPi0}
\Pi_0:= I_0^*I_0= {\pi_0}_*\Pi\,\pi_0^* : L^p(M)\to L^{p'}(M).
\end{equation}
Similarly, we define the self-adjoint bounded operator 
\begin{equation}\label{defofPi0e}
\Pi_0^e:= {\pi_0}_* \Pi^e \pi_0^*= (I^e \pi_0^*)^*I^e\pi_0^*:  L^p(M_e)\to L^{p'}(M_e).
\end{equation}
We first want to mention some boundedness result which holds in a general setting (no condition on conjugate points are required) and says that $\Pi_0$ is always regularizing if $V(t)$ decays sufficiently.
\begin{lemm}\label{boundednessI0}
Assume that \eqref{condescape} holds for $p>2$, then $I_0^*$ and $I_0$ are bounded as 
maps
\[ I_0^*: L^2(\pl_-SM,d\mu_\nu)\to H_{\rm loc}^{-\frac{n-1}{2}+\frac{n}{p}}(M^\circ), \quad 
I_0: H_{\rm comp}^{\frac{n-1}{2}-\frac{n}{p}}(M^\circ)\to L^2(\pl_-SM,d\mu_\nu).\]
and the same property holds for $I_0^e$ with $M_e$ replacing $M$. 
\end{lemm}
\begin{proof} It suffices to prove the boundedness for $I_0^*$. By Sobolev embedding,  
$I^*: L^2(\pl_-SM,d\mu_\nu)\to H_{\rm loc}^{-\frac{n}{2}+\frac{n}{p}}(M^\circ)$ is bounded,
and using Lemma \ref{I^*}, we have $XI^*=0$ as operators. Then applying 
\cite[Theorem 2.1]{GeGo} as in the proof of Proposition \ref{Rpm0}, we gain $1/2$ derivative in the Sobolev 
scale by applying ${\pi_0}_*$, this ends the proof.
\end{proof}
If $V(t)=\mc{O}(t^{-\infty})$, the Sobolev exponents are $H^{-1/2-\eps}_{\rm comp}(M^\circ)$ and 
$H^{1/2+\eps}_{\rm loc}(M^\circ)$ for all $\eps>0$, and if $K=\emptyset$ we get $I_0^*I_0: H^{-1/2}_{\rm comp}(M^\circ)\to H^{1/2}_{\rm loc}(M^\circ)$.
 Following the method of \cite{Gu}, we prove 
\begin{prop}\label{PsidoPi0}
Assume that the geodesic flow on $SM$ has no conjugate points and that the trapped set $K$ is hyperbolic.
The operator $\Pi_0^e={\pi_0}_* \Pi^e \pi_0^*$ is an elliptic pseudo-differential operator of order $-1$ in $M^\circ_e$, with principal symbol
$\sigma(\Pi_0^e)(x,\xi)=C_n|\xi|^{-1}_g$ for some constant $C_n\not=0$ depending only on $n$. 
\end{prop}
\begin{proof} First we choose the extension $(M_e,g)$ so that the geodesic flow on $M_e$ 
has non-conjugate points. 
Once we know the wavefront set of the Schwartz kernels of the resolvent $R_\pm(0)$, 
the proof is very similar to Theorem 3.1 and Theorem 3.4 in \cite{Gu}, therefore we do not write all details but refer to that paper where this is done carefully for Anosov flows. It suffices to analyze $\chi \Pi^e_0\chi'$ where $\chi,\chi'\in C_c^\infty(M_e^\circ)$ are arbitrary functions. Its Schwartz kernel 
is given by $\chi(x)\chi'(x')((\pi_0\otimes\pi_0)_*\Pi^e)(x,x')$ where $\Pi^e=R_+(0)-R_-(0)$ is identified with its Schwartz kernel. 
We write for $\eps\geq 0$ small
\[R_+(0)=\int_0^\eps e^{tX}dt + e^{\eps X}R_+(0)\]
where $e^{tX}$ is the pull-back by the flow at time $t$. Using \eqref{WFRpm} and the computation of ${\rm WF}(e^{\eps X})$ which follows from \cite[Theorem 8.2.4]{Ho},  the composition law of wavefront set \cite[Theorem 8.2.14]{Ho} can be used like in the proof of \cite[Theorem 3.1]{Gu}: we obtain  
\[\begin{split}
{\rm WF}(\pi_0^*(\chi) e^{\eps X}R_+(0)\pi_0^*(\chi'))\subset & \Big( \{(\varphi_t(y),(d\varphi_t(y)^{-1})^{T}\eta,y,-\eta);  \,\, t\leq -\eps,\,\, \eta(X(y))=0\}\\
&  \cup \{(\varphi_{-\eps}(y),\eta ,y,-d\varphi_{-\eps}(y)^T\eta);\,\, (y,\eta)\in T^*(SM)\setminus\{0\}\}\\
& \cup (E_-^*\x E_+^*) \Big) \cap \{ (y,\eta,y',\eta'); (\pi_0(y),\pi_0(y'))\in U\x U'\}.
\end{split}\]
where $U:=\supp(\chi)$ and $U'=\supp(\chi')$; here the wave-front set of an operator means the wave-front set of the Schwartz kernel of the operator. By applying the rule of pushforward of wave-front sets (given for example in \cite[Proposition 11.3.3.]{FrJo}), 
we get 
${\rm WF}({\pi_0}_*e^{\eps X}R_0\pi_0^*)\subset S_1 \cup S_2\cup S_3$ where 
\[\begin{split}
S_1:= \{ & (\pi_0(y),\xi,\pi_0(y'),\xi')\in T_0^*(U\x U); \,  (y,d\pi_0(y)^T\xi,y',d\pi_0(y')^T\xi')\in E_-^*\x E_+^*\}\\
S_2:=\{ & (\pi_0(\varphi_t(y)),\xi,\pi_0(y),\xi')\in T_0^*(U\x U); \,\, \exists \, t\leq -\eps,  \exists \, \eta, \eta(X(y))=0 , \\
            & \,\, d\pi_0(y)^T\xi'=-\eta, \,\, d\pi_0(\varphi_t(y))^T\xi=(d\varphi_t(y)^{-1})^{T}\eta\}\\
S_3:=\{( & \pi_0(\varphi_{-\eps}(y)),\xi,\pi_0(y),\xi')\in T_0^*(U\x U);\,\,  (d(\pi_0\circ \varphi_{-\eps})(y))^T\xi=-d\pi_0(y)^T\xi' \}
\end{split}\]
if we set $T_0^*(U\x U):=T^*(U\x U)\setminus \{0\}$. We let $V=\ker d\pi_0\subset T(SM_e)$ be the vertical bundle, and $H$ be the horizontal bundle (cf. \cite[Chapter 1.3]{Pa}), and $V^*,H^*\subset T^*(SM_e)$ their dual defined by $H^*(V)=0$ and $V^*(H)=0$ ($V^*$ is dual to $V$ and $H^*$ is dual to $H$ for the Sasaki metric).
By \eqref{klingenb}, the absence of conjugate points for the flow in $M_e$ implies that 
$T(SM_e)=\rr X\oplus V\oplus E_\pm$ at $\Gamma_\pm$ and thus $E_\pm^*\cap H^*=\{0\}$. This implies that 
$S_1=\emptyset$. Similarly, it is direct to see that $S_2=\emptyset$ is equivalent to the absence of conjugate points for the flow (see the proof of \cite[Theorem 3.1]{Gu} for details). 
The last part is $S_3$. The proof is exactly the same as in \cite[Theorem 3.1]{Gu} thus we do not repeat it but simply summarize the argument: the projection of $S_3$ on $M_e^\circ$
is contained in $\Delta_\eps(M_e^\circ\x M_e^\circ ):=\{(x,x')\in M_e^\circ\x M_e^\circ; d_g(x,x')=\eps\}$ where $d_g$ is the Riemannian distance. The operator 
$L_\eps=\int_0^\eps {\pi_0}_* e^{tX}\pi_0^*dt$ is explicit for small $\eps>0$ and given by 
\[ L_\eps f(x):= \int_{0}^\eps\int_{S_xM_e}f(\varphi_t(x,v))dvdt,\]
This operator has singular support $\Delta_\eps(M_e^\circ\x M_e^\circ)\cup \Delta_0(M_e^\circ\x M_e^\circ )$ and 
thus, $\eps>0$ being chosen arbitrary (but small), the kernel of $\Pi_0$ has singular support on the diagonal 
$\Delta_0(M_e^\circ\x M_e^\circ)$. Now the kernel $\psi(x,x')L_\eps(x,x')$ is that of an elliptic pseudo-differential operator of order $-1$ if $\psi\in C_c^\infty(M_e^\circ\x M_e^\circ)$ is supported close enough to the diagonal $\{x=x'\}$ and equal to $1$ in a neighborhood of the diagonal: the analysis is purely local and exactly the same as in \cite[Lemma 3.1]{PeUh},
which also shows that the symbol of this $\Psi$DO is $C_n|\xi|^{-1}_g$ for some $C_n>0$. It is direct to see 
(from $R_+(0)^*=-R_-(0)$) that $\Pi_0^e=2{\pi_0}_*R_+(0)\pi_0^*$, and we have then proved the claim. 
\end{proof}
Since the Schwartz kernel of $\Pi^e_0$ on $M^\circ$ is the restriction of the kernel of 
$\Pi^e$ to $M^\circ\x M^\circ$, we deduce that in the case of hyperbolic trapped set and no conjugate points, Lemma \ref{boundednessI0} gives that $\Pi^e_0:H^{-1/2}_{\rm comp}(M^\circ)\to H^{1/2}_{\rm loc}(M^\circ)$ and the $TT^*$ argument shows that for any compact domain $\mc{O}\subset M^\circ$ with non-empty interior and smooth boundary, we have
\begin{equation}\label{TT*}
I_0:H^{-1/2}(\mc{O})\to L^2(\pl_-SM;d\mu_\nu), \quad I_0^*: L^2(\pl_-SM;d\mu_\nu)\to H^{1/2}(\mc{O}).
\end{equation}
We can use Proposition \ref{PsidoPi0} to prove the regularity property on elements in $\ker I_0$. 
\begin{corr}\label{regulkerI}
Assume that the trapped set $K$ is hyperbolic, the metric has no conjugate points. 
Let $f_0\in L^p(M)+H^{-1/2}_{\rm comp}(M^\circ)$ for some $p>2$ satisfying $I_0f_0=0$. 
Then $f_0\in C^\infty(M)$ and $f_0$ vanishes to all order at $\pl M$.
\end{corr}
\begin{proof} First, $I_0f_0=0$ in $L^2(\pl_-SM; d\mu_{\nu})$ implies that $I^e_0f=0$ 
if $I^e_0=I^e\pi_0^*$ is the X-ray transform on functions on $M_e$ and $f_0$ is extended by $0$ in $M_e\setminus M$.
Thus $\Pi^e_0f_0=0$ in $M_e^\circ$. 
This implies, by ellipticity of $\Pi^e_0$ in $M_e^\circ$ that $f_0$ is smooth, and since it is equal to $0$ in $M_e^\circ\setminus M$, we deduce that $f_0$ vanishes to all order at 
$\pl M$. 
\end{proof}

\subsection{X-ray on symmetric tensors}
For any $m\in\nn$, symmetric cotensors of order $m$  on $M_e^\circ$ can be viewed as functions on $SM^\circ_e$ via the map 
\[\pi_m^*: C_c^{\infty}(M^\circ_e,\otimes_S^mT^*M^\circ_e)\to C_c^\infty(SM^\circ_e), \quad 
(\pi_m^*f)(x,v):= f(x)(\otimes^m v).\]
The dual operator is defined by
\[ {\pi_m}_*: C^{-\infty}(SM^\circ_e)\to C^{-\infty}(M^\circ_e,\otimes_S^mT^*M^\circ_e), \quad \cjg {\pi_m}_*u,f\cjd:= 
\cjg u,\pi_m^*f\cjd\]
Next, we define the operator $D:=\mc{S}\circ \nabla: C_c^{\infty}(M^\circ_e,\otimes_S^mT^*M_e)\to C_c^{\infty}(M^\circ_e,\otimes_S^{m+1}T^*M^\circ_e)$
by composing the Levi-Civita connection $\nabla$ with the symmetrization of tensors $\mc{S}:\otimes^{m+1}T^*M^\circ_e\to 
\otimes_S^{m+1}T^*M^\circ_e$.
The divergence of $m$-cotensors is the adjoint differential operator, which is given by $D^*f:=-\mc{T}(\nabla f)$ where  
$\mc{T}: \otimes_S^mT^*M\to  \otimes_S^{m-2}T^*M$ denotes the trace map defined by contracting with the Riemannian metric:
\begin{equation}\label{deftrace}  
\mc{T}(q)(v_1,\dots,v_{m-2}):=\sum_{i=1}^{n}q(e_i, e_i,v_1,\dots,v_{m-2})
\end{equation}
if $(e_1,\dots,e_n)$ is a local  orthonormal basis of $TM_e$. Each $u\in L^2(SM_e)$ 
function can be decomposed using the spectral decomposition of the vertical 
Laplacian $\Delta_v$ in the fibers of $SM_e$ (which are spheres)
\begin{equation}\label{decompuk}
u=\sum_{k=0}^\infty u_k, \quad \Delta_vu_k=k(k+n-2).
\end{equation}
where $u_k$ are $L^2$ sections of a vector bundle over $M_e$; see \cite{GuKa2, PSU}.

When \eqref{condescape} holds for some $p\in (2,\infty)$, we
define just as for $m=0$ the X-ray transform on $\otimes_S^mT^*M$ as the bounded operator
for all $p\in(2,\infty)$
\begin{equation}\label{defIm}
I_m := I\, \pi_m^*: L^p(M; \otimes_S^mT^*M)\to L^2(\pl_-SM, d\mu_\nu).
\end{equation} 
The adjoint $I_m^*:L^2(\pl_-SM, d\mu_\nu)\to L^{p'}(M; \otimes_S^mT^*M)$ is bounded if 
$1/p'+1/p=1$ and it is given by $I_m^*={\pi_m}_*I^*$. The operator $\Pi_m$ 
is simply defined as the bounded self-adjoint operator for $p\in (2,\infty)$ and $1/p'+1/p=1$ 
\begin{equation}\label{defofPim}
\Pi_m:= I_m^*I_m= {\pi_0}_*\Pi\,\pi_m^* : L^p(M; \otimes_S^mT^*M)\to L^{p'}(M; \otimes_S^mT^*M).
\end{equation}
As for $m=0$, we set $\Pi^e_m:={\pi_0}_*\Pi^e\,\pi_m^*$, which can also be seen as $(I^e_m)^*I^e_m$ on 
if $I^e_m=I^e\pi_m^*$ is the X-ray transform on $m$ cotensors on $M_e$. 
Repeating the arguments of \cite[Theorem 3.5]{Gu} but adapted to our case we get directly  
\begin{prop}\label{Psido}
Assume that the geodesic flow on $M$ has no conjugate points and that the trapped $K$ is hyperbolic.
For $m\geq 1$, 
the operator $\Pi^e_m$ is a pseudo-differential operator of order $-1$ on the bundle 
$\otimes_S^mT^*M^\circ_e$, which is elliptic on $\ker D^*$ in the sense that for all $\psi_0\in C_c^\infty(SM_e^\circ)$ 
there exist pseudo-differential operators $Q,S,R$ on 
$M^\circ_e$ with respective order $1,-2,-\infty$ so that 
\begin{equation}\label{parametrix}
Q\psi_0\Pi^e_m\psi_0=\psi_0^2+D\psi_0 S\psi_0D^*+R 
\end{equation}
\end{prop}
The only difference with \cite[Theorem 3.5]{Gu} is that the flow is not hyperbolic everywhere anymore, but using that the bundle 
$E_\pm^*$ are transverse to the annihilator $H^*$ of the vertical bundle $V=\ker d\pi_0$,
the proof reduces to be the same, just as we explained in the proof of Proposition \ref{PsidoPi0} for $m=0$. We do not 
repeat the arguments, as it does not bring anything new. The same result as \eqref{TT*} also holds for 
$I_m$ and $I_m^*$ since $\Pi_m$ is a $\Psi$DO of order $-1$: if $\mc{O}\subset M^\circ$ is any compact domain (with non-empty interior) with smooth boundary,
\begin{equation}\label{TT*m}
I_m:H^{-1/2}(\mc{O},\otimes^m_ST^*M)\to L^2(\pl_-SM;d\mu_\nu).
\end{equation}

\subsection{Injectivity of X-ray transform on symmetric tensors}

In this section, we use the Pestov identity and the smoothness property in Corollary \ref{regulkerI}
to prove injectivity of X-ray transform on functions and 1-forms in case of hyperbolic trapping. The proof is basically the same as in the simple domain setting, once we have proved the smoothness of elements in $\ker I_m\cap \ker D^*$.

\begin{theo}\label{Inj}
Let $(M,g)$ be a compact Riemannian manifold with strictly convex boundary. Assume that  the geodesic flow has no conjugate points, that the trapped set $K$ is hyperbolic. \\ 
1) Let $f_0\in L^p(M)+H^{-1/2}_{\rm comp}(M^\circ)$ with $p>2$ such that $I_0f_0=0$, then $f_0=0$.\\
2) Let $f_1\in C^\infty(M; T^*M)+H^{-1/2}_{\rm comp}(M^\circ;T^*M)$ such that $I_1f_1=0$, 
then there exists $\psi\in C^\infty(M)+H^{1/2}_{\rm comp}(M^\circ)$ vanishing at $\pl M$ such that $f_1=d\psi$.\\
3) Assume that the sectional curvatures of $g$ are non-positive, then if for $m>1$, $f_m\in C^\infty(M; \otimes_S^mT^*M)$  satisfies $I_mf_m=0$, then $f_m=Dp_{m-1}$ for some $p_{m-1}\in C^\infty(M; \otimes_S^{m-1}T^*M)$ which vanishes at $\pl M$.
\end{theo}
\begin{proof} Let us first show 1) and 2). Using Hodge decomposition we write $f_1=d\psi+f_1'$ with 
$f_1'\in C^\infty(M,T^*M)+H^{-1/2}_{\rm comp}(M^\circ,T^*M)$ satisfying $D^*f_1'=0$ and 
$\psi\in C^\infty(M)+H^{1/2}_{\rm comp}(M^\circ)$ satisfying $\psi|_{\pl M}=0$. This can be done by taking $\psi:=\Delta_D^{-1}\delta f_1$ where $\Delta_D^{-1}$ is the inverse of the Dirichlet Laplacian  on 
$(M,g)$ and $\delta:=d^*=D^*$ on $1$-forms. Notice that $f_1'$ is smooth near $\pl M$ since $f_1$ is 
(using ellipticity of $\Delta_D$).
Since $I_1d\psi=0$ we get $\Pi_1f'_1=0$ and $\Pi_1^ef_1'=0$. 
By applying \eqref{parametrix} to $f_1'$ with $\psi_0=1$ on $M$, 
we get that $f_1'\in C^{\infty}(M^\circ)$ thus $f_1'\in C^{\infty}(M)$.
Since also $\Pi_0f_0=0$, Corollary \ref{regulkerI} then implies that $f_0$ and $f_1'$ are smooth.
By Proposition \ref{kernelPi}, we see that there exists $u_j\in C^\infty(SM)$ for $j=0,1$ 
such that $Xu_0=\pi_0^*f_0$ and $Xu_1=\pi_1^*f_1'$, with 
$u_j$ vanishing to all order on $\pl SM$. Now since the functions $u_j$ are smooth and vanish at the boundary $\pl SM$, Pestov's identity \cite[Proposition 2.2. and Remark 2.3]{PSU} holds here in the same way as it does for simple manifolds with boundary or for closed manifolds: 
\begin{equation}\label{pestov} 
|| \nabla^v Xu_j ||^2_{L^2}=||X\nabla^vu_j||^2_{L^2}-\cjg R\nabla^v u_j,\nabla^vu_j\cjd +(n-1)||Xu_j||_{L^2}^2
\end{equation}
where $\nabla^v$ is the covariant derivative in the vertical direction of $SM$, mapping functions on 
$SM$ to sections of the bundle $E\to SM$ with fibers 
\[ E_{(x,v)}:=\{ w\in T_xM; g_x(w,v)=0\},\]
$R$ is the curvature tensor acting on $E$ by $R_{(x,v)}w:=R(w,v)v\in E_{(x,v)}$, and $X$ acts on sections of $E$
by differentiating parallel transport along the geodesic (see Section 2 of \cite{PSU}).
Then the proof of Lemma 11.2 of \cite{PSU} and Proposition 7.2 of \cite{DKSU} 
is based on Santalo's formula \eqref{santalo} and thus
applies as well in our setting (ie. the boundary is strictly convex, there is   
no conjugate points and $\Gamma_+\cup \Gamma_-$ has Liouville measure $0$), 
then for all $Z\in C^\infty(SM,E)$
\[ ||XZ||_{L^2}-\cjg RZ,Z\cjd\geq 0\]
with equality if and only if $Z=0$. In particular, since $\nabla^v Xu_0=\nabla^vf_0=0$, 
we deduce from \eqref{pestov} that $f_0=0$, and since $||\nabla^vXu_1||_{L^2}^2=(n-1)||f_1||^2_{L^2}$, we deduce 
from \eqref{pestov} that $\nabla^vu_1=0$ and thus $u_1=\pi_0^*\psi'$ for some smooth function 
$\psi'$ on $M$ which vanishes to all order at $\pl M$; 
this implies that $Xu_1=\pi_1^*d\psi'$. Notice that if $D^*f_1'=0$, then $D^*f_1'=\Delta_g\psi'=0$ and therefore $\psi'=0$ since $\psi'$ vanishes at $\pl M$. Thus $f_1'=0$.

Finally, the case with $m>1$ when the curvature of $g$ is non-positive uses the proof
 of \cite{CrSh} (in the closed case) and \cite[Section 11]{PSU} (in the case of simple domains). 
If $I_mf=0$, we also have $I^e_mf_m=0$ and thus $\Pi^e\pi_m^*f_m=0$. By Proposition \ref{kernelPi}, 
there exists $u=-R_+(0)\pi_m^*f_m=-R_-(0)\pi_m^*f_m$ smooth in $SM$ such that $Xu=\pi_m^*f_m$
and $u|_{\pl SM}=0$.  Non-positive curvature implies that the flow is $1$-controlled in the sense of \cite{PSU} and once we know that 
$Xu=\pi_m^*f_m$ with $u$ smooth and vanishing at $\pl M$, the proof of Theorem 11.8 in \cite{PSU} 
(that proof is detailed in Section 9 and 11) based on Pestov identity applies verbatim in our case . We do not repeat it here as it does not bring anything new. 
\end{proof}

We get Theorem \ref{Th1} and Theorem \ref{Th0bis} as a direct corollary:

\emph{Proof of Theorem \ref{Th1}}. We only prove 2) since the 
conformal case 1) is easier and a direct consequence of point 1) in Theorem \ref{Inj}.
If the metrics are lens equivalent, $\Gamma_\pm\cap \pl_\pm SM$ are the same for all metrics, and
for a fixed $y:=(x,v)\in \pl_-SM\setminus \Gamma_-$, the geodesic $\gamma_s(y;t)$ with $t\in [0,\ell_+(y)]$ depends smoothly on 
$s$ (by general ODE arguments) and by differentiating $\pl_s\ell_+(y)^2=0$, we obtain
that $q_s:=\pl_s g_s$ is a smooth symmetric $2$-tensors satisfying $I_2^sq_s=0$ if $I_2^s$ is the X-ray for $g_{s}$ on symmetric $2$ cotensors. The argument is standard and detailed in \cite[Section 1.1]{Sh}. Applying Theorem \ref{Inj} with $m=2$ in non-positive curvature shows 
that $q_s=D_sp_s$ for some smooth $1$-form $p_s$ vanishing at $\pl M$. The tensor $p_s$ can be written 
as $p_s=(\Delta_{D_s})^{-1}D_s^*q_s$ if $\Delta_{D_s}:=D_s^*D_s$ with Dirichlet condition at $\pl M$ 
(this is invertible, see \cite{Sh}). Then we argue like in the proof of \cite[Theorem 1]{GuKa1}: 
by ellipticity of $\Delta_{D_s}$ and smootness in $s$, $p_s$ is smooth in $s$.
Then one can construct a smooth family of diffeomorphisms $\phi_s$ which are the identity on $\pl M$ so that 
$\phi_s^{-1}\pl_s\phi_s=p_s$ and $\phi_0={\rm Id}$ (here we view $p_s$ as a vector field). This concludes the proof.
\qed

\emph{Proof of Theorem \ref{Th0bis}}. A negatively curved manifold with strictly convex boundary has hyperbolic trapped set $K$ 
(see \cite[\S3.9 and Theorem 3.2.17]{Kl2}), no conjugate points (see \cite{Kl}). Thus, combining Theorem \ref{Th1} 
with Proposition \ref{youngbowenruelle}, we obtain Theorem \ref{Th0bis}.\qed

\subsection{Invariant distributions with prescribed push-forward}

We will show the existence of invariant distributions on $SM$ with prescribed push-forward. This corresponds essentially 
to surjectivity of $I^*_0$ and of $I_1^*$ on $\ker D^*$. 
\begin{prop}\label{surj}
We make the same assumptions as in Theorem \ref{Inj}.\\ 
1) For any $f_0\in H^s(M)$ for $s>1$, there exists 
$w\in (\cap_{u<0} H^{u}(SM_e))\cap L^1(SM_e)$ such that 
$Xw=0$ in $SM_e^\circ$ and ${\pi_0}_*w=f_0$ in $M$. Moreover, if $f_0\in C^\infty(M)$, $w$ 
 has wavefront set satisfying
$ {\rm WF}(w)\subset E_+^*\cup E_-^*$
and its boundary value $\omega=w|_{\pl SM}$ satisfies 
\eqref{wboundary} and $\omega\in L^2_{S_g}(\pl SM)$, and $w\in H^{s}(SM_e)$ for some $s>0$.\\
2) Let $f_1\in C^\infty(M; T^*M)$ satisfying $D^*f_1=0$, then there exists $w\in L^{p'}(SM_e)$ 
such that $Xw=0$ in $SM_e^\circ$ and ${\pi_1}_*w=f_1$ in $M$, with 
${\rm WF}(w)\subset E_+^*\cup E_-^*$ and  $\omega:=w|_{\pl SM}$ satisfies 
\eqref{wboundary} and  is in $L^2_{S_g}(\pl SM)$.
\end{prop}
\begin{proof} 
Let $Y$ be a closed manifold extending smoothly $M_e$ across its boundary,  
 extend the metric smoothly to $Y$ (and still call the extension $g$). 
Let $\psi_0\in C^\infty_c(Y)$ with support in $M_e$ which is equal to $1$ on a neighborhood of $M$ and write 
$\psi:=\pi_0^*(\psi_0)$ its lift to $SY$. 
Using Proposition \ref{PsidoPi0}, define the elliptic $\Psi$DO of order $-1$ on $Y$
\[P_0=\psi_0 \Pi^e_0 \psi_0+(1-\psi_0)(1+\Delta_g)^{-1/2}(1-\psi_0) : H^{-s}(Y)\to H^{-s+1}(Y)\]
bounded for all $s\geq 0$; here $\Delta_g$ is the Laplacian on $(Y,g)$. 
Thus there exists $C>0$ and $K:H^{-s}(Y)\to H^{-s+1}(Y)$ 
a bounded $\Psi$DO (of order $-1$) such that for all $f\in H^{-s}(Y)$
\[ ||P_0f||_{H^{1-s}(Y)}\geq C||f||_{H^{-s}(Y)}-||Kf||_{H^{-s+1}(Y)}\]
and thus the range of $P_0$ is closed.  Consequently, by Banach closed range theorem, 
$P_0^*:H^{s-1}(Y)\to H^{s}(Y)$ has closed range. Note that $P_0^*$ has the same form as $P_0$, 
and to prove its surjectivity, it suffices to prove injectivity of $P_0$. 
If $P_0f=0$, then $f\in C^\infty(Y)$ by ellipticity of $P_0$, and $(1-\psi_0)f=0$ since $(1+\Delta_g)^{-1/2}$ is injective,
 and $\cjg \Pi^e_0(\psi_0f),\psi_0f\cjd_{L^2}=0$. This implies that $I_0^e(\psi_0f)=0$ and by 
Theorem \ref{Inj} applied with $M_e$ instead of $M$, we get $\psi_0f=0$, thus $f=0$.
We deduce that if $f_0\in H^s(M)$, taking an extension $\tilde{f}_0\in H^{s}(Y)$ supported 
in the region where $\psi_0=1$, there exists a unique $u\in H^{s-1}(Y)$ such that $P_0^*u=\tilde{f}_0$. Note that if 
$f_0$ is smooth, $u$ is smooth by ellipticity of $P_0^*$.
In particular, we get $\psi_0\Pi^e_0(\psi_0u)=\tilde{f}_0$ and taking $w:=\Pi^e(\psi_0u)$, we get $Xw=0$ in $SM_e$,
${\pi_0}_*w=f_0$ in $M$, and by Proposition \ref{Pi}, we obtain the desired regularity
for $w$ and the properties of its restriction $w|_{\pl SM}$ and \eqref{wboundary}. This proves 1). 

The proof of 2) is essentially the same as in \cite[Lemma 2.2]{DaUh} once we know 
Proposition \ref{Psido} and the kernel of $I_1$. We just recall very briefly the argument 
and refer to \cite[Lemma 2.2]{DaUh} for details. First, by  \cite[Corollary 3.3]{KMPT} 
(see also the last remark of that paper for the manifold case) there is a bounded extension operator 
$E: \ker D^*|_{L^2(M,T^*M)}\to  \ker D^*|_{L^2(M_e^\circ,T^*M_e)}$ which restricts continuously
to $E:\ker D^*|_{C^\infty(M,T^*M)}\to  \ker D^*|_{C_c^\infty(M_e^\circ,T^*M_e)}$ then if 
$r_M:L^2(M_e,T^*M_e)\to L^2(M,T^*M)$ is the restriction to $M$, we get from Proposition \ref{Psido} 
that $r_M\Pi_1^e\psi_0Q^*E={\rm Id}+r_MR^*E$ as a map on $\ker D^*|_{L^2(M,T^*M)}$ 
with $R$ smoothing on $M_e^\circ$. This implies that the range of ${\rm Id}+r_MR^*E$ is closed with finite codimension, and the same holds on $\ker D^*|_{C^\infty(M,T^*M)}$. 
Then $r_M\Pi_1^e\psi_0Q^*E(\ker D^*|_{C^\infty(M,T^*M)})$ has closed range in $\ker D^*|_{C^\infty(M,T^*M)}$
with finite codimension and thus $r_M\Pi_1^e\psi_0Q^*(C_0^\infty(M_e^\circ,T^*M_e))$ has closed range with finite codimension in $\ker D^*|_{C^\infty(M,T^*M)}$. The kernel of the adjoint is trivial by using 
Theorem \ref{Inj} just as in \cite[Lemma 2.2.]{DaUh}. This shows that there is $u\in C^\infty(M_e,T^*M_e)$ such that $r_M\Pi_1^eu=f_1$, and thus setting $w:=\Pi^e\pi_1^*u$ we get the result.
 \end{proof}

\section{Determination of the conformal structure for surfaces}

In this Section, we will study the lens rigidity for surfaces with strictly convex boundary, no conjugate points 
and hyperbolic trapped set. To recover the conformal structure from the scattering map, we shall use most of the results proved above together with the approach of Pestov-Uhlmann \cite{PeUh} which reduces the scattering rigidity to the Calder\'on problem on surfaces. 

For the oriented Riemannian surface $M_e$ with boundary, the unit tangent bundle $SM_e$ is a principal circle bundle, with an action 
\[ S^1\x SM_e\to SM_e , \quad e^{i\theta}.(x,v)=(x,R_\theta v)\]
where $R_\theta$ is the rotation of angle $+\theta$. This induces a vector field $V$ generating this action, defined by
$Vf(x,v)=\pl_\theta(f(e^{i\theta}.(x,v))|_{\theta=0}$. We then define the vector field $X_\perp:=[X,V]$ and the basis 
$(X,X_\perp,V)$ is an orthonormal basis of $SM_e$ for the Sasaki metric.
The space $SM_e$ splits into $SM_e=\mc{V}\oplus\mc{H}$ where 
$\mc{V}=\rr V=\ker d\pi_0$ is the vertical space, and $\mc{H}={\rm span}(X,X_\perp)$ the horizontal space which can also be defined using the Levi-Civita connection (see for example \cite{Pa}). 
Following Guillemin-Kazhdan \cite{GuKa1}, there is an orthogonal decomposition (Fourier series in the fibers) 
\begin{equation}\label{decomp}
L^2(SM^\circ_e)=\bigoplus_{k\in\zz} \Omega_k, \quad \textrm{ with }Vw_k=ikw_k \textrm{ if }w_k\in \Omega_k
\end{equation}
where $\Omega_k$ is the space of $L^2$ sections of a complex line bundle over $M_e^\circ$. 
Similarly, one has a decomposition on $\pl SM$
\begin{equation}\label{decompplSM}
L^2(\pl SM)=\bigoplus_{k\in\zz} \Omega_k', \quad \textrm{ with }V\omega_k=ik\omega_k \textrm{ if }\omega_k\in \Omega'_{k}
\end{equation}
using Fourier analysis in the fibers of the circle bundle.

\subsection{Hilbert transform and Pestov-Uhlmann commutator relation}
The Hilbert transform in the fibers is defined by using the decomposition \eqref{decomp}: 
\[H: L^2(SM^\circ_e)\to L^2(SM^\circ_e),\quad H(\sum_{k\in\zz}w_k)=-i\sum_{k\in\zz}{\rm sign}(k)w_k.\]
with ${\rm sign}(0):=0$ by convention. It is skew-adjoint and $\bbar{Hu}=H\bbar{u}$, thus we can extend continuously 
$H$ to $C^{-\infty}(SM^\circ_e)\to C^{-\infty}(SM^\circ_e)$ by the expression
\[ \cjg Hu,\psi\cjd :=-\cjg u,H\psi\cjd , \quad \psi\in C_c^\infty(SM^\circ_e) \]
where the distribution pairing is $\cjg u,\psi\cjd=\int_{SM_e}u\psi d\mu$ when $u\in L^2(SM^\circ_e)$.
Similarly, we define the Hilbert transform in the fibers on $\pl SM$
\[H_\pl: C^\infty(\pl SM)\to C^\infty(\pl SM),\quad H_\pl(\sum_{k\in\zz}\omega_k)=-i\sum_{k\in\zz}{\rm sign}(k)\omega_k\]
and its extension to distributions as for $SM_e$.
For smooth $w\in C_c^\infty(SM_e^\circ)$ we have that 
\begin{equation}\label{hilbertrest} 
(Hw)|_{\pl SM}=H_{\pl}\,\omega, \quad\textrm{ with } \omega:=w|_{\pl SM}
\end{equation}
thus the identity extends by continuity to the space of distributions in $SM_e^\circ$ with wave-front set
disjoint from $N^*(\pl SM)$ since, by \cite[Theorem 8.2.4]{Ho}, the restriction map $C^\infty(SM_e^\circ)\to C^\infty(\pl SM)$ 
obtained by pull-back through the inclusion map $\iota$ of \eqref{iota} extends continuously to the space 
of distributions on $SM_e^\circ$ with wavefront set not intersecting $N^*(\pl SM)$.
By \cite[Lemma 3.5]{Gu}, we see that ${\rm WF}(Hu)\subset {\rm WF}(u)$ for all $u\in C^{-\infty}(SM_e)$ and the 
same holds for $H_\pl$ and $u\in C^{-\infty}(\pl SM)$.
The following commutator relation between Hilbert transform and flow follows easily from the Fourier decomposition and was proved by Pestov-Uhlmann \cite[Theorem 1.5]{PeUh}: 
\begin{equation}\label{commutator}
\textrm{ if }w\in C^\infty(SM_e^\circ), \quad [H,X]w=X_\perp w_0+(X_\perp w)_0
\end{equation}
where $w_0=\frac{1}{2\pi}\pi_0^*({\pi_0}_*w)$ and
${\pi_0}_*w(x)=\int_{S_xM_e}w(x,v)dS_x(v)$ for smooth $w$. 
Notice that $w\in C^\infty(SM_e^\circ)\mapsto w_0\in C^\infty(SM_e^\circ)$ 
extends continuously to $C^{-\infty}(SM_e^\circ)$ since 
$\pi_0$ is a submersion (the pullback $\pi_0^*$ extends to distributions), 
then the relation \eqref{commutator} extends continuously to 
$C^{-\infty}(SM_e^\circ)$. We also have, for any $w\in C^{-\infty}(SM_e^\circ)$ 
\begin{equation}\label{Xpu0}
X_\perp w_0= \frac{1}{2\pi}\pi_1^*(*d ({\pi_0}_*w)).
\end{equation}
where $*: T^*M_e\to T^*M_e$ is the Hodge-star operator on 1-forms.
We use the odd/even decomposition of distributions with respect to the involution $A(x,v)=(x,-v)$ on 
$SM_e$, $SM$ and $\pl SM$, as explained in the end of Section \ref{subsec:resolvent}.
The operator $X$ maps odd distributions to even distributions and conversely.
The operator $H$ maps odd (resp. even) distributions to odd (resp. even) distributions, we set 
$H_{\rm ev} w:=H(w_{\rm ev})$ and $H_{\rm od}w:=H(w_{\rm od})$. We write similarly 
$H_{\pl,{\rm ev}}$ and $H_{\pl,{\rm od}}$ for the Hilbert transform on (open sets of) $\pl SM$ and the relation \eqref{hilbertrest} also holds with $H_{\pl, {\rm ev}}$ replacing $H_{\pl}$ if $w$ is even.
Taking the odd part of \eqref{commutator}, we have for any $w\in C^{-\infty}(SM_e^\circ)$ 
\begin{equation}\label{oddpart}
\begin{gathered}
H_{\rm od}Xw-XH_{\rm ev}w=\frac{1}{2\pi}\pi_1^*(*d ({\pi_0}_*w))=X_\perp w_0.
\end{gathered}
\end{equation}

\subsection{Determination of the conformal structure from scattering map}

For functions $\omega\in C^{\infty}(\pl SM)$, the function ${\pi_0}_*\omega$ is smooth on $\pl M$, given by the expression ${\pi_0}_*\omega(x,v)=\frac{1}{2\pi}\int_{S_xM_e}\omega(x,v)dS_x(v)$ and 
thus if $w\in C^{\infty}(SM_e^\circ)$ and $\omega=w|_{\pl SM}$, one has 
${\pi_0}_*\omega= ({\pi_0}_*w)|_{\pl M}.$
 As above, the restriction map $C^\infty(SM_e^\circ)\to C^\infty(\pl SM)$, extends continuously to the space 
of distributions on $SM_e^\circ$ with wavefront set included in $E_+^*\cup E_-^*$ 
(since this does not intersect $N^*(\pl SM)$). Therefore,  for 
$w\in C^{-\infty}(SM_e^\circ)$ with ${\rm WF}(w)\subset E_+^*\cup E_-^*$, we have
\begin{equation}\label{restpi0}
{\pi_0}_*\omega= ({\pi_0}_*w)|_{\pl M}, \quad \textrm{ with }\omega:=w|_{\pl SM}
\end{equation}
in the distribution sense (in fact, as in the proof of Proposition \ref{surj}, it is easily checked that 
${\pi_0}_*w\in C^\infty(M_e^\circ)$). 

For an oriented Riemannian surface $(M,g)$ with boundary, the space of holomorphic functions
can be described as follows: $f=f_1+if_2$ is holomorphic if $*df_1=df_2$ where $*$ is the Hodge star operator.
We use the notation $\mc{P}(f)\in C^\infty(M)$ for the unique solution of $\Delta_g \mc{P}(f)=0$  with $\mc{P}(f)=f$ 
on $\pl M$. 

\begin{theo}\label{scatrig}
Let $(M,g)$ and $(M',g')$ be two oriented Riemannian surfaces with the same boundary 
$N$, and $g|_{TN}=g'|_{TN}$. 
For both surfaces, assume that the boundary is strictly convex,  
the trapped set are hyperbolic, that \eqref{condescape} holds, and the metrics have no conjugate points.
If $(M,g)$ and $(M',g')$ are scattering equivalent, then there exists a diffeomorphism 
$\phi:M\to M'$ with $\phi|_{\pl M}={\rm Id}$ and such that $\phi^*g'=e^{2\eta}g$ for some $\eta\in C^\infty(M)$ satisfying 
$\eta|_{\pl M}=0$. 
\end{theo}
\begin{proof}  
We shall follow the method of Pestov-Uhlmann \cite{PeUh} and we will need to use most of the results from the previous sections.
We work on $(M,g)$ but all the results below 
apply as well on $(M',g')$. For $f\in C^\infty(N)$, the harmonic extension $\mc{P}(f)$ admits a harmonic conjugate 
$\mc{P}(f^*)$ if  $*d\mc{P}(f)=d\mc{P}(f^*)$ or equivalently $\mc{P}(f+if^*)$ is holomorphic.
We are going to prove the following statement: let $f^*\in C^\infty(N)$,  then 
\begin{equation} \label{identiteS-Id}
2\pi(S_{g}^*-{\rm Id})(H_{\pl,{\rm ev}}\omega)=(S_{g}^*-{\rm Id})\pi_0^*f^*
\end{equation}
holds for some $\omega\in L^2_{S_g}(\pl SM)$ satisfying ${\rm WF}(\omega_-)\subset E_{\partial,-}^*$ and ${\rm WF}(\omega_+)\subset E_{\partial,+}^*$, if and only if
\begin{equation}\label{reciproq}
I_0^*\omega_-=\mc{P}(f) \textrm{ with }\mc{P}(f-if^*) \textrm{ holomorphic}
\end{equation}
where ${\pi_0}_*\mc{E}_-=I_0^*$ (see Lemma \ref{I^*}) and $\omega_\pm:=\omega|_{\pl_\pm SM}$. 

Let us prove the first sense. Let $f\in C^\infty(N)$ so that $\mc{P}(f)$ admits a harmonic conjugate. 
Using Proposition \ref{surj}, there exists $w\in L^1(SM_e)\cap C^\infty(SM_{e}\setminus (\Gamma_+\cup \Gamma_-))$ satisfying $Xw=0$ in $M^\circ_{e}$ in the distribution sense 
with ${\pi_0}_*w=\mc{P}(f)$ in $M$ and 
\begin{equation}\label{restriction}
\begin{gathered}
\omega:=w|_{\pl SM}\in L^2_{S_g}(\pl SM),\quad 
{\rm WF}(\omega)\subset E_{\pl,+}^*\cup 
E_{\pl,-}^*
\end{gathered}\end{equation} 
$\omega_{-}:=\omega|_{\pl_-SM}$, 
where $E_{\pm,\pl}^*\subset T_{\Gamma_\pm}^*(\pl SM)$ 
are the bundles defined by \eqref{Epmbound}  for the manifold $M$ and  ${\pi_0}_*$ 
is the pushforward defined by \eqref{pushforward} on $SM$.  
From \eqref{oddpart} and using that $H_{\rm ev}w$ is smooth in $SM\setminus (\Gamma_-\cup\Gamma_+)$, we get 
\begin{equation}\label{XH}
XH_{\rm ev}w= -\frac{1}{2\pi}\pi_1^*(*d\mc{P}(f))
\end{equation}
as smooth functions on $SM\setminus (\Gamma_-\cup\Gamma_+)$.
 Now,  for any $\psi\in C^\infty (SM\setminus (\Gamma_+\cup \Gamma_-))$,
\[IX\psi=(S_{g}^*-{\rm Id})(\psi|_{\pl SM \setminus (\Gamma_-\cup \Gamma_+)})\] 
as a function on $\pl_-SM\setminus \Gamma_-$. 
Applying $I$ to \eqref{XH} and using that $\mc{P}(f-if^*)$ is holomorphic then gives ($I_1$ is the X-ray transform on $1$-forms)
\[
2\pi (S_{g}^*-{\rm Id})((H_{\rm ev}w)|_{\pl SM})=-I_1(*d\mc{P}(f))=I_1(d\mc{P}(f^*))=IX\pi^*_0(\mc{P}(f^*))= (S_{g}^*-{\rm Id})\pi_0^*f^*
 \]
as smooth functions on $\pl_-SM\setminus \Gamma_-$ 
which are globally in $L^2(\pl_-SM,d\mu_\nu)$. Using \eqref{hilbertrest} we thus obtain the identity \eqref{identiteS-Id}.

Next, we prove the converse. Conversely, let $f^*\in C^\infty(N)$, 
let $q\in C^\infty(M)$ with $q|_{\pl M}=f^*$ and let $\chi\in C_c^\infty(SM^\circ)$ which is equal to 
$1$ in $\{\rho>\eps\}$ with $\eps>0$ small (using $\rho$ as in Section \ref{subset:extension}), thus on $K$.
We write $w_1:=\chi\mc{E}_-\omega_-$ and $w_2:=(1-\chi)\mc{E}_-\omega_-$ and  
by \eqref{oddpart}, we get for $j=1,2$
\begin{equation}\label{decompoi} 
HXw_j-XHw_j=\pi_1^*(*d{\pi_0}_*w_j).
\end{equation}
By Proposition \ref{boundval}, ${\rm WF}(w_2)\subset E_+^*\cup E_-^*$ thus ${\pi_0}_*w_2\in C^\infty(M)$ (using $(E^*_-\cup E_+^*)\cap H^*=\{0\}$ if $H^*\subset T^*(SM_e^\circ)$ is the annulator of the vertical bundle $V=\ker d\pi_0$), and ${\pi_0}_*w_1\in H_{\rm comp}^{1/2}(SM_e^\circ)$ with support containing $K$. 
We claim that we can apply $I$ to \eqref{decompoi} and view the result as a measurable function in 
$\pl_-SM\setminus \Gamma_-$:  for $j=2$ we can apply $I$ since all terms are smooth in 
$SM\setminus (\Gamma_-\cup\Gamma_+)$ and we get a smooth function 
on $\pl_-SM\setminus \Gamma_-$ that is in $L^2(\pl_-SM)$ and for $j=1$ the only possible trouble is $I_1(*d{\pi_0}_*w_1)$ but this makes sense 
since $I_1: H_{\rm comp}^{-1/2}(M^\circ,T^*M)\to L^2(\pl_-SM,d\mu_\nu)$ is bounded just as $I_0$ 
in \eqref{TT*} (see the remark after Proposition \ref{Psido}). Therefore, applying $I$ to \eqref{decompoi} and summing  
for $j=1,2$, we obtain almost everywhere on $\pl_-SM$
\[(S_{g}^*-{\rm Id})(H_{\pl,{\rm ev}}\omega)=IXH\mc{E}_-(\omega_-)
=-\frac{1}{2\pi}I_1(*d{\pi_0}_*w_1+*d{\pi_0}_*w_2),\]
this term is in $L^2(\pl_-SM,d\mu_\nu)$ and equal to $\frac{1}{2\pi}(S_{g}^*-{\rm Id})\pi_0^*f^*=\frac{1}{2\pi}I_1(dq)$ by our assumption. Since we know that this term is smooth on $\pl_-SM$ we obtain in $L^2(\pl_-SM,d\mu_\nu)$ 
\[I_1(*dI_0^*\omega_-+dq)=0.\]
By Theorem \ref{Inj} one has $*dI_0^*\omega_-+dq=d\psi$ for some $\psi\in C^{\infty}(M)+H_{\rm comp}^{1/2}(M^\circ)$ satisfying 
$\psi|_{\pl M}=0$. Applying first $d$ and then $d*$ to that equation and using ellipticity, we get 
$\psi-q\in C^\infty(M)$ and $I_0^*\omega_-\in C^\infty(M)$ and both functions are harmonic conjugate, 
which means that \eqref{reciproq} holds with $f:=(I_0^*\omega_-)|_{\pl M}$.  

We can finally finish the proof. All that we said above applies also on $(M',g')$ and we shall put  
prime for objects related to $g'$.
Let $\alpha:SM'\to SM$ be the map \eqref{alpha}, so that $\alpha\circ S_{g'}=S_{g}\circ \alpha$ by assumption. 
Remark that for each $\omega\in C^\infty(\pl SM)$, 
$(\omega\circ \alpha)_k=\omega_k\circ\alpha$ in the Fourier 
decomposition  \eqref{decompplSM}, and thus 
\begin{equation}\label{alpha^*}
\alpha^*(H_{\pl,{\rm ev}}\omega)=H'_{\pl,{\rm ev}}(\alpha^*\omega).
\end{equation}
This identity extends to $\omega\in L^2(\pl SM)$ by continuity. Let $f^*\in C^\infty(N)$ and assume 
that there exists $f\in C^\infty(N)$ so that $\mc{P}(f+if^*)$ is holomorphic in $(M,g)$, then we have proved that there is 
$\omega\in L^2_{S_g}(\pl SM)$ satisfying \eqref{identiteS-Id}, ${\pi_0}_*\omega=f$ and \eqref{restriction}.
Using $\alpha\circ S_{g'}=S_{g}\circ \alpha$ and $\pi_0\circ \alpha=\pi_0$,  
together with \eqref{alpha^*}, we get 
\begin{equation}\label{identiteinV1V2}
(S_{g'}^*-{\rm Id})(H'_{\pl, {\rm ev}}\omega')=(S_{g'}^*-{\rm Id})\pi_0^*f^*. 
\end{equation} 
with $\omega':=\alpha^*\omega$. We can use Lemma \ref{SgdetermE} which implies that 
${\rm WF}(\omega')\subset E_{\pl,+}^{'\, *}\cup E^{'\, *}_{\pl,-}$, and since $\omega'\in L^2_{S_g}(SM')$, 
we get by \eqref{reciproq} applied with $(M',g')$ that ${I_0'}^*(\omega')-i\mc{P}'(f^*)$ is holomorphic in 
$(M',g')$. Since ${I_0'}^*(\omega')|_{\pl M}={\pi_0}_*\omega=f$, we have shown that 
all boundary value of a holomorphic function on $(M,g)$ is also the boundary value of one on $(M',g')$. Exchanging the role of $(M,g)$ and $(M',g')$, we show that the space of boundary values of holomorphic functions on 
$(M,g)$ and $(M,g')$ are the same. The existence of the conformal diffeomorphism $\phi:M\to M'$ then follows from the 
work of Belishev \cite{Be}. 
\end{proof}

\end{document}